\documentclass[openacc,a4paper]{rsproca_new}

\usepackage{amsmath}
\usepackage{amsfonts}
\usepackage{hyperref}
\usepackage{cleveref}
\usepackage{graphicx}
\usepackage{epstopdf}
\ifpdf
\DeclareGraphicsExtensions{.eps,.pdf,.png,.jpg}
\else
\DeclareGraphicsExtensions{.eps}
\fi
\usepackage{bm}
\usepackage[caption=false,font=footnotesize,indention=0cm]{subfig}
\usepackage{array}
\usepackage{xcolor}

\usepackage[utf8]{inputenc} 
\usepackage[export]{adjustbox}
\usepackage{pgfplots}
\usepackage{pgfplotstable}
\usepgfplotslibrary{external}
\usepackage[shortlabels]{enumitem}
\usepackage{algorithm}
\usepackage{algpseudocode}
\usepackage{multirow}

\usepackage{enumitem}
\setlist[enumerate]{leftmargin=.5in}
\setlist[itemize]{leftmargin=.5in}

\def\N{\mathbb{N}}
\def\R{\mathbb{R}}

\def\dir{\texttt{DGFM3}}

\def\aiuse#1{{\vskip5.5pt\noindent \textcolor{jobcolor}{\fontsize{9}{11}\selectfont Declaration of AI use.}\fontsize{8}{11}\selectfont\enskip #1}}

\makeatletter
\newcommand\footnoteref[1]{\protected@xdef\@thefnmark{\ref{#1}}\@footnotemark}
\makeatother

\def\addlegendimage{\csname pgfplots@addlegendimage\endcsname}

\def\first#1{{\setlength{\fboxsep}{2pt}\fcolorbox{black}{black!15}{\textbf{#1}}}}
\def\second#1{{\setlength{\fboxsep}{2pt}\fcolorbox{black!50}{black!5}{#1}}}

\newtheorem{definition}{\bf Definition}[section]
\newtheorem{proposition}{\bf Proposition}[section]
\newtheorem{remark}{\bf Remark}[section]

\titlehead{Research}

\begin{document}
	
	\title{Gradient flow-based modularity maximization for community detection in multiplex networks}
	
	\author{
		Kai Bergermann$^{1}$ and Martin Stoll$^{1}$}
	
	\address{$^{1}$Department of Mathematics, Chemnitz University of Technology, 09107 Chemnitz, Germany}
	
	\subject{
		applied mathematics, computational mathematics
	}
	
	
	\corres{Martin Stoll\\
		\email{martin.stoll@mathematik.tu-chemnitz.de}}
	
	\begin{abstract}
		We propose two methods for the unsupervised detection of communities in undirected multiplex networks.
		These networks consist of multiple layers that record different relationships between the same entities or incorporate data from different sources.
		Both methods are formulated as gradient flows of suitable energy functionals: the first (\texttt{MPBTV}) builds on the minimization of a balanced total variation functional, which we show to be equivalent to multiplex modularity maximization, while the second (\dir) directly maximizes multiplex modularity.
		The resulting non-linear matrix-valued ordinary differential equations (ODEs) are solved efficiently by a graph Merriman--Bence--Osher (MBO) scheme.
		Key to the efficiency is the approximate integration of the discrete linear differential operators by truncated eigendecompositions in the matrix exponential function.
		Numerical experiments on several real-world multiplex networks show that our methods are competitive with the state of the art with respect to various metrics.
		Their major benefit is a significant reduction of computational complexity leading to runtimes that are orders of magnitude faster for large multiplex networks.
	\end{abstract}
	



\begin{fmtext}
\end{fmtext}
\maketitle

\section{Introduction}\label{sec:intro}

 Many complex systems can be modeled by graphs or networks\footnote{We use the terms graph and network synonymously throughout the manuscript.} that record pairwise relations between its entities.
 Albeit their conceptual simplicity, the study of networked complex systems has led to a rich body of work analyzing structural and dynamical properties \cite{newman2003structure,boccaletti2006complex,estrada2012structure}.
 Recent years have witnessed the trend in network science to move beyond pairwise interactions or single layers and instead employ generalized graph models such as hypergraphs, simplicial complexes, or multilayer networks \cite{kivela2014multilayer,boccaletti2014structure,battiston2020networks,battiston2021physics,bick2023higher}.

 In this work, we consider undirected and possibly weighted multiplex networks in which edges within layers represent different relationships between the same set of nodes or incorporate data from different sources and edges across layers are restricted to connecting instances of the same entity in different layers \cite{kivela2014multilayer,boccaletti2014structure}.
 In recent years, many well-studied techniques for single-layer networks have been extended to multiplex or multilayer networks.
 One common difficulty in generalizing methods for the structural analysis of multilayer networks such as community detection, centrality analysis, or core-periphery detection is that traditional techniques may obtain conflicting results on the different network layers.
 For instance, an individual in a social multilayer network in which layers represent different social media platforms may be connected to different people in different layers and hence be assigned to different communities.
 Similarly, in temporal networks community affiliation may change over time.
 One current line of work to circumvent this issue is the automated detection of informative and noisy layers as well as the optimization of weights that characterize the importance of layers \cite{venturini2023learning,bergermann2024nonlinear}.
 There are, however, examples in the literature illustrating that information is lost when aggregating multilayer networks into single-layer networks, cf., e.g., \cite[Section VI.A.]{bergermann2022fast}.
 An alternative approach that we pursue in this work is to perform the structural analysis on the node-layer pair level, i.e., compute centrality or coreness scores for each node-layer pair or allow physical nodes to change community labels across layers \cite{mucha2010community,de2015identifying,boutemine2017mining,kuncheva2015community,taylor2021tunable,bergermann2022fast,bergermann2024core}.
 
 \subsection{Background}
 
 The unsupervised detection of communities has been one of the the most-studied problems in network science and graph-based learning in the past two decades.
 Although different definitions, objectives, and quality functions exist in the literature, a community, cluster, or module is generally referred to as a strongly connected subgraph of the overall network.
 In the early $2000$s, Girvan and Newman observed a hierarchical organization of community structure in real-world networks, proposed several algorithms for their detection, and introduced a scalar measure called modularity to quantify the quality of a given community partition \cite{girvan2002community,newman2003structure,newman2004finding,newman2006finding,newman2006modularity}.
 The proposed algorithms were based on removing central edges or computing eigenvectors of certain matrix representations of the network.
 Other community detection methods developed later were based on, e.g., random walks \cite{pons2006computing,rosvall2008maps} or performing k-means clustering on the eigenvectors to the smallest eigenvalues of the graph Laplacian matrix \cite{von2007tutorial}, cf.~\cite{fortunato2010community} for a survey as of the year $2010$.
 
 The assignment of nodes to communities generally represents a combinatorial optimization problem that is infeasible to solve exactly even for moderate numbers of communities and network entities.
 Moreover, in the absence of \emph{a priori} knowledge, the number of communities represented in a network is typically unknown making its determination part of the problem.
 Although proven to be strongly NP-complete \cite{brandes2007modularity}, one successful class of community detection methods in single-layer networks aims at maximizing the quality measure modularity, cf.~\Cref{rem:single-layer_modularity}, that was initially proposed to determine the ``best'' partition from a hierarchy of partitions on different resolution levels, which involves the choice of the ``best'' number of communities \cite{newman2004finding}.
 In these methods, the combinatorial problem is typically tackled by heuristic approaches \cite{blondel2008fast,mucha2010community,traag2019louvain,venturini2022variance} or relaxed formulations \cite{hu2013method,boyd2018simplified,tudisco2018community,zhang2018sparse,cristofari2020total}.
 
 Such relaxed formulations often rely on elegant equivalences.
 For example, it has been shown that single-layer modularity coincides with the auto-covariance function of random walks on networks at certain times \cite{delvenne2010stability,lambiotte2014random,lambiotte2021modularity}.
 Furthermore, maximizing modularity is equivalent to maximum likelihood estimation in planted partition stochastic block models (SBMs) \cite{newman2016equivalence}.
 The authors of \cite{cristofari2020total} show the equivalence of modularity maximization and total variation minimization and realize the latter by a non-linear active-set optimization approach.
 Some years earlier, the equivalence of modularity maximization with certain graph cut problems and balanced total variation minimization has been shown by Hu et.\ al.\ \cite{hu2013method} and Boyd et.\ al.\ \cite{boyd2018simplified}.
 They proposed the modularity Merriman--Bence--Osher (MBO) scheme \cite{hu2013method} and the Pseudospectral Balanced total variation (BTV) MBO scheme \cite{boyd2018simplified} to minimize suitable Ginzburg--Landau-type energy functionals via gradient flows.
 These approaches transform the community detection problem into Allen--Cahn-type matrix-valued ordinary differential equations (ODEs) that have also been studied in the context of image inpainting and semi-supervised learning on graphs \cite{bertozzi2012diffuse,hu2013method,garcia2014multiclass,boyd2018simplified,budd2021classification,bergermann2021semi}.
 
 The generalization of single-layer community detection methods to multilayer networks started with Mucha et.\ al.'s seminal work \cite{mucha2010community} proposing a definition of modularity for time-dependent, multiscale, and multiplex networks that gives rise to a generalized version of the single-layer Louvain method \cite{blondel2008fast} implemented in the GenLouvain software package \cite{jutla2011generalized}.
 In subsequent years, various other techniques for community detection in multiplex and multilayer networks based on random walks \cite{de2015identifying,kuncheva2015community}, label propagation \cite{boutemine2017mining}, SBMs \cite{taylor2016enhanced,pamfil2019relating}, spectral clustering \cite{mercado2018power}, and versions of GenLouvain \cite{bazzi2016community,venturini2022variance} have been proposed in the literature, cf.~the two survey articles \cite{magnani2021community,huang2021survey}.
 
 \subsection{Contribution}
 
 In this work, we propose two approaches for community detection in multiplex networks that are based on minimizing suitable energy functionals via gradient flows.
 The first method that we call multiplex balanced total variation (\texttt{MPBTV}) generalizes the single-layer BTV method \cite{boyd2018simplified} and builds on the minimization of a balanced total variation functional, which we show to be equivalent to multiplex modularity maximization.
 The second method directly maximizes multiplex modularity by means of a gradient flow.
 We discuss the efficient solution of the resulting matrix-valued ODEs by a graph MBO scheme with truncated eigendecompositions of the respective discrete linear differential operator.
 Numerical experiments illustrate that our methods are competitive with the state of the art in terms of multiplex modularity, classification accuracy, and NMI.
 Their main advantage over state-of-the-art methods is a significant reduction of computational complexity leading to runtimes that are orders of magnitude lower for large multiplex networks.
 
 \subsection{Outline}
 
 The remainder of the manuscript is structured as follows.
 \Cref{sec:multiplex_networks} introduces multiplex networks alongside their linear algebraic representation.
 \Cref{sec:modularity} introduces multiplex modularity as an objective for community detection in multiplex networks.
 The two methods proposed in this manuscript are introduced in \Cref{sec:MPBTV,sec:dir}.
 \Cref{sec:MPBTV} includes the equivalence proof of multiplex modularity maximization and balanced total variation minimization and introduces the means to solving the latter by the minimization of a suitable Ginzburg--Landau-type functional.
 We discuss the efficient numerical solution of the resulting matrix-valued ODEs in \Cref{sec:numerical_solution_of_ODEs} before presenting numerical experiments on multiplex networks of small to large scale in \Cref{sec:numerical_experiments}.

\section{Multiplex networks}\label{sec:multiplex_networks}

 In this work, we consider node-aligned, layer-coupled, undirected, and possibly weighted multiplex networks $\mathcal{G} = (\widetilde{\mathcal{V}}, \mathcal{E}^{(1)}, \dots, \mathcal{E}^{(L)}, \widetilde{\mathcal{E}})$.
 Here, $\widetilde{\mathcal{V}} = \{1,\dots,n\}$ denotes the set of $|\widetilde{\mathcal{V}}| = n$ consecutively numbered physical nodes representing the network's entities that are present in each of the $L$ layers of the multiplex network \cite{kivela2014multilayer}.
 We also define the set $\mathcal{V} = \{ (j,\ell) : j=1,\dots,n,~\ell=1,\dots,L \}$ of $|\mathcal{V}| = nL$ node-layer pairs representing the instances of physical nodes in the different layers as well as $\mathcal{V}^{(\ell)} = \{ (j,\ell) : j=1,\dots,n \}$ representing each physical node in layer $\ell$.
 
 The layers of the multiplex network may record different types of interactions between the physical nodes or incorporate data from different sources that should be treated separately in order to prevent a loss of information that often occurs in aggregation, cf., e.g., \cite{kivela2014multilayer,boccaletti2014structure,de2015identifying,bergermann2022fast}.
 Pairwise connections between the node-layer pairs within layers are recorded in the intra-layer edge sets $\mathcal{E}^{(\ell)}\subset\mathcal{V}^{(\ell)}\times \mathcal{V}^{(\ell)}$ for $\ell=1,\dots,L$.
 In contrast to general multilayer networks that additionally allow arbitrary connections between node-layer pairs from different layers, multiplex networks only possess layer-coupled inter-layer edges $\widetilde{\mathcal{E}}$, i.e., edges across layers between node-layer pairs representing the same physical node $i\in\widetilde{\mathcal{V}}$.
 
 Many tools from single-layer network analysis are based on linear algebraic network representations in terms of adjacency or graph Laplacian matrices.
 Multiplex networks $\mathcal{G}$ possess multiple representations in terms of matrices or tensors and in this work we adopt the framework based on supra-adjacency and supra-Laplacian matrices \cite{de2013mathematical,kivela2014multilayer,boccaletti2014structure}.
 \begin{definition}\label{def:multiplex_network_la}
 	The symmetric indefinite matrices $\bm{A}^{(\ell)}\in\R_{\geq 0}^{n\times n}$ represent the intra-layer adjacency matrices of layers $\ell=1,\dots,L$ in which a non-zero entry $\bm{A}^{(\ell)}_{ij}$ denotes the presence of an edge between nodes $i$ and $j$ in layer $\ell$.
 	Furthermore, the diagonal matrices $\bm{C}^{(k\ell)}\in\R_{\geq 0}^{n\times n}$ with $k,\ell=1,\dots,L$ and $k\neq \ell$ record the inter-layer edge weights between all node-layer pairs in layers $k$ and $\ell$.
 	We restrict ourselves to layer-coupled multiplex networks in which we can write $\bm{C}^{(k\ell)} = \omega\widetilde{\bm{A}}_{k\ell}\bm{I}$, where $\bm{I}\in\R^{n\times n}$ denotes the identity matrix, $\widetilde{\bm{A}}_{k\ell}$ is the corresponding entry in the symmetric layer coupling matrix $\widetilde{\bm{A}}\in\R_{\geq 0}^{L\times L}$ that has zero entries on the diagonal, and the inter-layer coupling parameter $\omega\geq 0$.
 	
 	With these ingredients, we can define the supra-adjacency matrix
	\begin{align}\label{eq:supra-adjacency}
	\begin{split}
 	\bm{A} & := \begin{bmatrix}
 	\bm{A}^{(1)} & \bm{C}^{(12)} & \dots & \bm{C}^{(1L)}\\
 	\bm{C}^{(21)} & \bm{A}^{(2)} & \dots & \bm{C}^{(2L)}\\
 	\vdots & \vdots & \ddots & \vdots\\
 	\bm{C}^{(L1)} & \bm{C}^{(L2)} & \dots & \bm{A}^{(L)}
 	\end{bmatrix}\in\R_{\geq 0}^{nL \times nL}\\
 	& =: \bm{A}_{\mathrm{intra}} + \omega \bm{A}_{\mathrm{inter}} := \mathrm{blkdiag}\left[ \bm{A}^{(1)}, \dots, \bm{A}^{(L)} \right] + \omega \widetilde{\bm{A}}\otimes \bm{I},
 	\end{split}
 	\end{align}
 	where the inter-layer coupling parameter $\omega\geq 0$ allows to trade off the relative weighting of intra- and inter-layer edges that is generally symmetric and indefinite.
 	
 	\begin{sloppypar}
 	For each layer $\ell=1,\dots,L$, we define the intra-layer degree vector \mbox{$\bm{d}^{(\ell)} = \bm{A}^{(\ell)}\bm{1}_n\in\R^n$}, where $\bm{1}_n\in\R^n$ denotes the vector of all ones.
 	With this, we define the individual layer graph Laplacians $\bm{L}^{(\ell)} = \mathrm{diag}(\bm{d}^{(\ell)}) - \bm{A}^{(\ell)}$ for $\ell=1,\dots,L$.
 	The degree vector of the supra-adjacency matrix reads
 	\begin{equation*}
 	\bm{d} = \bm{A1}_{nL} = \begin{bmatrix}(\bm{d}^{(1)})^T & \dots & (\bm{d}^{(L)})^T\end{bmatrix}^T + \omega\widetilde{\bm{A}}\bm{1}_L\otimes \bm{1}_n\in\R^{nL}.
 	\end{equation*}
 	Finally, the supra-Laplacian is defined as the graph Laplacian corresponding to the supra-adjacency matrix \eqref{eq:supra-adjacency} that reads
 	\begin{align}\label{eq:supra-laplacian}
 	\begin{split}
 	\bm{L} & = \mathrm{diag}(\bm{d}) - \bm{A} = \bm{L}_{\mathrm{intra}} + \omega\bm{L}_{\mathrm{inter}}\in\R_{\geq 0}^{nL \times nL}\\
 	& = \mathrm{blkdiag}\left[ \bm{L}^{(1)}, \dots, \bm{L}^{(L)} \right] + \omega \left( \mathrm{diag}(\widetilde{\bm{A}}\bm{1}_L) - \widetilde{\bm{A}} \right) \otimes \bm{I}.
 	\end{split}
 	\end{align}
 	\end{sloppypar}
 \end{definition}

\section{Community detection by modularity maximization}\label{sec:modularity}

 One of the most-studied problems in the structural analysis of networks and graph-based learning is the unsupervised detection of communities in networks.
 A community generally refers to a subgraph that is more strongly connected internally than to the overall network.
 The task of community detection algorithms is thus to find a partition of the network's entities that is suitable with respect to some objective function or quality measure.
 
 In general, the number of communities in a network is unknown or subject to the desired resolution, cf.~\Cref{rem:resolution_parameter}, and its determination is part of the community detection problem.
 We first fix notation to denote partitions of the node-layer pairs of a multiplex network into communities.
 \begin{definition}\label{def:sets_community_affiliation}
 	Let $n_c$ denote the number of distinct communities in a non-overlapping partition of the set of node-layer pairs of a multiplex network into subsets $S_r\subset \mathcal{V}$ with $r=1,\dots,n_c$, $S_r\cap S_t = \emptyset$ for $r\neq t$, and $S_1 \cup\dots\cup S_{n_c} = \mathcal{V}$.
 	Furthermore, let $S^{(\ell)}_r = S_r \cap \mathcal{V}^{(\ell)}, r=1,\dots,n_c$ denote the set of node-layer pairs of $S_r$ present in layer $\ell=1,\dots,L$.
 	 	
 	We define the characteristic community affiliation function
 	\begin{equation*}
 	U: \{1,\dots,n\} \times \{1,\dots,L\} \rightarrow \{1,\dots,n_c\}
 	\end{equation*}
 	that maps each node-layer pair $(j,\ell)$ to its community number $1,\dots,n_c$.
 	The community affiliation of all node-layer pairs can be expressed by the binary matrix
 	\begin{equation*}
 	 \bm{U} = \begin{bmatrix}
 	 \bm{U}^{(1)}\\
 	 \bm{U}^{(2)}\\
 	 \vdots\\
 	 \bm{U}^{(L)}
 	 \end{bmatrix}\in\{0,1\}^{nL \times n_c}, \qquad \bm{U}^{(\ell)}\in\{0,1\}^{n \times n_c},
 	\end{equation*}
 	in which the one-hot encoded row $(\ell-1)n+j$ corresponding to node-layer pair $(j,\ell)$ equals the standard basis vector $\bm{e}_c^T\in\R^{n_c}$ when $(j,\ell)$ belongs to community $c\in\{1,\dots,n_c\}$.
 	The ordering of node-layer pairs in $\bm{U}$ corresponds to that in the supra-adjacency matrix \eqref{eq:supra-adjacency} such that $\bm{U}^{(\ell)}$ represents the community affiliation of all node-layer pairs in layer $\ell$.
 \end{definition}

 There is a multitude of objective functions in the literature that quantify the quality of community partitions in different ways.
 One very popular such measure is modularity.
 It was originally introduced by Newman and Girvan \cite{newman2004finding} to identify the ``best'' partition from a hierarchy of partitions on different resolution levels.
 It compares the edge weights occurring in the adjacency matrix of a network to those that would be expected under a prescribed null model.
 Large modularity values are attained for partitions that assign nodes to the same community that mainly have stronger edge weights between them than expected at random while mainly having weaker-than-expected edge weights across communities.
 
 The modularity of multiplex networks consists of two separate contributions of intra- and inter-layer edges.
 The idea of comparing existing edge weights to a null model is applied to the intra-layer edges of each layer.
 In addition, positive contributions of inter-layer edges within communities promote the assignment of instances of the same physical node in different layers to the same community.
 We state modularity for multiplex networks as introduced by Mucha et.\ al.\ \cite{mucha2010community} and give the original single-layer version in \Cref{rem:single-layer_modularity}.
 \begin{definition}\label{def:multiplex_modularity}
	The modularity of a multiplex network is defined as
	\begin{equation}\label{eq:multiplex_modularity}
	Q(U) = \frac{1}{2\mu} \sum_{i,j=1}^n \sum_{k,\ell=1}^L \left[ \left( \bm{A}_{ij}^{(\ell)} - \frac{\gamma^{(\ell)}}{2m^{(\ell)}} \bm{d}_i^{(\ell)}\bm{d}_j^{(\ell)} \right) \delta_{k\ell} + \bm{C}_{j}^{(k\ell)}\delta_{ij} \right] \delta(U(i,k),U(j,\ell)),
	\end{equation}
	where $\delta$ denotes the Kronecker delta, $2m^{(\ell)} = \bm{1}_n^T\bm{d}^{(\ell)} = \bm{1}_n^T\bm{A}^{(\ell)}\bm{1}_n$ the intra-layer strength of layer $\ell=1,\dots,L$, $2\mu = \bm{1}_{nL}^T \bm{A}\bm{1}_{nL}$ the total strength of the multiplex network, and $\gamma^{(\ell)}\in\R_{>0}$ the resolution parameters, cf.~\Cref{rem:resolution_parameter}.
 \end{definition}
 
 For later use in this manuscript it is beneficial to compactly express the multiplex modularity $Q$ in the following linear algebra notation.
 The claim can be verified by a straightforward computation.

 \begin{proposition}
 	\begin{sloppypar}
	For matrix-valued community affiliation notation, i.e., \mbox{$\bm{U}\in\{0,1\}^{nL \times n_c}$} with \mbox{$\bm{U1}_{n_c}=\bm{1}_{nL}$}, cf.~\Cref{def:sets_community_affiliation}, we can write \eqref{eq:multiplex_modularity} as
	\end{sloppypar}
	\begin{equation}\label{eq:multiplex_modularity_la}
	Q(\bm{U}) = \frac{1}{2\mu} \mathrm{tr}\left(\bm{U}^T\bm{M}\bm{U}\right),
	\end{equation}
	where the block matrix
	\begin{equation*}
	\bm{M} = \begin{bmatrix}
	\bm{M}^{(1)} & \bm{C}^{(12)} & \dots & \bm{C}^{(1L)}\\
	\bm{C}^{(21)} & \bm{M}^{(2)} & \dots & \bm{C}^{(2L)}\\
	\vdots & \vdots & \ddots & \vdots\\
	\bm{C}^{(L1)} & \bm{C}^{(L2)} & \dots & \bm{M}^{(L)}
	\end{bmatrix}\in\R^{nL\times nL}, \qquad \bm{M}^{(\ell)} = \bm{A}^{(\ell)} -  \frac{\gamma^{(\ell)}}{2m^{(\ell)}}\bm{d}^{(\ell)}(\bm{d}^{(\ell)})^T,
	\end{equation*}
	defines the multiplex modularity matrix and $\bm{M}^{(\ell)}\in\R^{n\times n}$ the single-layer modularity matrices of layers $\ell=1,\dots,L$, cf.~\Cref{rem:single-layer_modularity}, and $\mathrm{tr}$ denotes the trace of a matrix.
 \end{proposition}

 We make a couple of remarks on the definition of multiplex modularity.
 
 \begin{remark}\label{rem:single-layer_modularity}
	For single-layer networks, i.e., $L=1$, and all layer indices dropped, multiplex modularity \eqref{eq:multiplex_modularity} simplifies to
	\begin{equation}\label{eq:single-layer_modularity}
	Q(\bm{U}) = \frac{1}{2m} \mathrm{tr}\left(\bm{U}^T \bm{M}\bm{U}\right) = \frac{1}{2m} \sum_{i,j=1}^n \left( \bm{A}_{ij} - \gamma \frac{\bm{d}_i\bm{d}_j}{2m} \right) \delta(U(i), U(j)),
	\end{equation}
	for $2m=2\mu=\bm{1}_n^T\bm{A1}_n$ and the single-layer modularity matrix $\bm{M} = \bm{A} - \frac{\gamma}{2m}\bm{dd}^T$.
 \end{remark}

 \begin{remark}\label{rem:resolution_parameter}
 	The original version of modularity for single-layer networks \cite{newman2004finding} did not contain the scalar resolution parameter $\gamma>0$ that was later introduced by Reichardt and Bornholdt \cite{reichardt2006statistical}.
 	It has, however, become a standard ingredient since it has been shown that community detection methods based on modularity maximization suffer from the so-called resolution limit \cite{fortunato2007resolution}.
 	It describes the failure of this class of methods to resolve communities below a certain size that depends on the network size and the degree of interconnectedness of the communities.
 	The choice $\gamma>1$ leads to a larger number of smaller communities compared to the case $\gamma=1$ since it assigns more weight to the negative contribution of the null model while $\gamma<1$ yields fewer large communities.
 \end{remark}
 \begin{remark}
	The terms $\frac{\bm{d}_i^{(\ell)}\bm{d}_j^{(\ell)}}{2m^{(\ell)}}$ in \eqref{eq:multiplex_modularity} and $\frac{\bm{d}_i\bm{d}_j}{2m}$ in \eqref{eq:single-layer_modularity} represent the Chung--Lu random graph model \cite{chung2002average,chung2002connected} in which the expected edge weight between two nodes is given by the normalized product of their degrees.
	It is the most frequently used null model in modularity maximization although other choices have been proposed in the literature \cite{arenas2008analysis,ronhovde2010local,traag2011narrow,fasino2016generalized}.
 \end{remark}
 \begin{remark}
 	Throughout this manuscript, we consider all-to-all layer coupling represented by the matrix $\widetilde{\bm{A}} = \bm{11}^T - \bm{I}$, cf.~\Cref{def:multiplex_network_la}.
 	Different choices may be preferable in other situations, e.g., if the presence or absence of inter-layer edges reflects the possibility of changing between modes of transportation at a given node when evaluating walk-based network measures for transport systems, cf.~\cite{bergermann2021orientations,bergermann2024core}.
 	In modularity maximization for multiplex networks, however, it appears sensible to connect all representations of a physical node since it promotes their assignment to the same community.
 	As an extreme example, consider a node-layer pair $(j,\ell)$ that is isolated within layer $\ell$ such that $\bm{A}_{ij}^{(\ell)} - \frac{\gamma^{(\ell)}}{2m^{(\ell)}} \bm{d}_i^{(\ell)}\bm{d}_j^{(\ell)}=0$ for all $i=1,\dots,n$.
 	The only way for it to contribute to multiplex modularity \eqref{def:multiplex_modularity} is via non-zero inter-layer edge weights $\bm{C}_{j}^{(k\ell)}$ to node-layer pairs $(j,k)$ in other layers $k=1,\dots,L,$ where $(j,k)$ may be part of a strongly connected community.
 \end{remark}
 
 Although originally introduced as a scalar measure for the quality of a given partition \cite{newman2004finding}, the maximization of (multiplex) modularity has emerged as the objective underlying a wide class of community detection methods, cf., e.g., \cite{newman2006modularity,brandes2007modularity,blondel2008fast,mucha2010community,hu2013method,zhang2018sparse,boyd2018simplified,traag2019louvain}.
 In particular, the Louvain method \cite{blondel2008fast} and its generalization to multiplex networks (GenLouvain) \cite{mucha2010community} enjoy huge popularity.
 The Louvain method represents a greedy approach that initializes every node with its own community, considers each node in a randomized order, and merges it into the community of the neighboring node that improves modularity the most.
 In the case of multiplex networks, the procedure is performed on the node-layer pair level.
 Each detected community is condensed into a super-node when no more merging with positive gain in modularity is possible and the process is repeated recursively.
 
 Performing each step in the Louvain method locally optimal with respect to increasing modularity makes it challenging for competing methods to produce partitions that obtain higher modularity scores.
 However, alternative metrics such as classification accuracy or normalized mutual information (NMI) score exist in situations in which a ground truth community structure is available (at least at one point of a hierarchy of partitions and with respect to selected node features such as department affiliation in a social network).
 Since it has been observed that the modularity landscape of a given network often possesses multiple local maxima corresponding to quite different partitions with different accuracies and NMIs \cite{good2010performance,boyd2018simplified} it may be worthwhile to also consider the latter two metrics.

\section{Multiplex balanced total variation minimization}\label{sec:MPBTV}

 In this and the next section, we propose two methods for unsupervised community detection in multiplex networks that are based on maximizing multiplex modularity, cf.~\Cref{def:multiplex_modularity}, by means of gradient flows of suitable energy functionals.
 As discussed in \Cref{sec:intro}, the aim of both methods is to assign a community label to each node-layer pair of a multiplex network for a fixed number of communities $n_c\in\N$.
 The number of communities is thus a hyper-parameter in both methods that does not need to be specified \emph{a priori} in other community detection methods.
 As described in more detail in \Cref{sec:numerical_solution_of_ODEs}, our methods overcome this limitation by low runtimes that allow the detection of the number $n_c$ of communities via grid search when $n_c$ is relatively small.
 For larger numbers of communities, a recursive splitting approach as discussed for the single-layer BTV method \cite{boyd2018simplified} could be applicable.
 
 The first method proposed in this manuscript is built on the single-layer methods introduced in \cite{hu2013method,boyd2018simplified}.
 They leverage an elegant equivalence between the maximization of single-layer modularity, cf.~\Cref{rem:single-layer_modularity}, a certain graph cut problem, and a total variation minimization formulation.
 The latter is realized by numerically minimizing a Ginzburg--Landau-type energy functional.
 In classical continuum models, this process models phase separation phenomena with diffuse interfaces whose width is proportional to an interface parameter $\epsilon>0$.
 The functional's $\Gamma$-convergence to a total variation functional considers the sharp interface limit $\epsilon\rightarrow 0$, which promotes the separation of nodes into clearly separated communities in the discrete graph setting, cf.~\cite{van2012gamma} and references therein.
 Note that a similar equivalence of modularity maximization and total variation minimization for single-layer networks has been shown in \cite{cristofari2020total}, where the authors propose a non-linear active-set optimization approach.
 
 Our first method generalizes the single-layer Balanced TV approach proposed by Boyd et.\ al.\ \cite{boyd2018simplified} to multiplex networks and we call it multiplex balanced total variation (\texttt{MPBTV}).
 Before we can state similar equivalences between multiplex modularity maximization, multiplex graph cuts, and the minimization of total variation functionals, we require the following definitions.
 \begin{definition}
 	We define the volume of subsets $S_r$ and $S_r^{(\ell)}$ of node-layer pairs defined in \Cref{def:sets_community_affiliation} via
 	\begin{equation*}
 	 \mathrm{vol}(S_r) = \sum_{(j,\ell)\in S_r} \bm{d}_{(\ell-1)n+j}, \qquad \mathrm{vol}(S^{(\ell)}_r) = \sum_{j\in S^{(\ell)}_r} \bm{d}^{(\ell)}_{j}.
 	\end{equation*}
 	Note that the index $(\ell-1)n+j$ corresponds to node-layer pair $(j,\ell)$ in the framework introduced in \Cref{def:multiplex_network_la}.
 	Moreover, we denote the complement of the set $S_r$ by $S_r^C = \mathcal{V}\setminus S_r$, which allows us to define the graph cut of a multiplex network as the sum of edge weights between node-layer pairs across the subsets $S_r$ and $S_r^C$, i.e.,
 	\begin{equation*}
 	 \mathrm{Cut}_{\mathrm{MP}}(S_r,S_r^C) = \sum_{(i,k)\in S_r, (j,\ell)\in S_r^C} \bm{A}_{(k-1)n+i,(\ell-1)n+j}.
 	\end{equation*}
 	Finally, we define the multiplex total variation of a partition of node-layer pairs into $S_1,\dots,S_{n_c}$ as the usual network total variation of the supra-adjacency matrix $\bm{A}\in\R^{nL \times nL}$ without differentiating between intra- and inter-layer edges as
 	\begin{align*}
 	|\bm{U}|_{\mathrm{MPTV}} & = \frac{1}{2} \sum_{i,j=1}^n \sum_{k,\ell=1}^L \bm{A}_{(k-1)n+i,(\ell-1)n+j} | U(i,k) - U(j,\ell) |\\
 	& = \sum_{r=1}^{n_c} \sum_{(i,k)\in S_r, (j,\ell)\in S_r^C} \bm{A}_{(k-1)n+i,(\ell-1)n+j} = \sum_{r=1}^{n_c} \mathrm{Cut}_{\mathrm{MP}}(S_r,S_r^C).
 	\end{align*}
 \end{definition}
 
 The above generalizations of community affiliation, graph cuts, and total variation for multiplex networks allow us to state the following result in analogy to the single-layer results \cite[Theorem 2.1]{hu2013method} and \cite[Proposition 3.4]{boyd2018simplified}.
 
 \begin{proposition}\label{prop:equivalence}
	The following optimization problems are equivalent.
	\begin{enumerate}
		\item Maximizing multiplex modularity \eqref{eq:multiplex_modularity_la}
		\begin{equation}\label{eq:multiplex_mod_max}
		 \max_{\substack{\bm{U}\in\{0,1\}^{nL \times n_c}\\\bm{U1}=\bm{1}}} Q(\bm{U})
		\end{equation}
		\item Minimizing the balanced multiplex graph cut
		\begin{equation*}
		 \min_{\substack{\bm{U}\in\{0,1\}^{nL \times n_c}\\\bm{U1}=\bm{1}}} \sum_{r=1}^{n_c}\mathrm{Cut}_{\mathrm{MP}}(S_r,S_r^C) + \sum_{r=1}^{n_c}\sum_{\ell=1}^L \frac{\gamma^{(\ell)}}{2m^{(\ell)}} \sum_{i\in S_r^{(\ell)}} \bm{d}_{(\ell-1)n+i}\mathrm{vol}(S_r^{(\ell)})
		\end{equation*}
		\item Minimizing the balanced multiplex total variation functional
		\begin{equation}\label{eq:multiplex_BTV}
		 \min_{\substack{\bm{U}\in\{0,1\}^{nL \times n_c}\\\bm{U1}=\bm{1}}} |\bm{U}|_{\mathrm{MPTV}} + \sum_{\ell=1}^L \frac{\gamma^{(\ell)}}{2m^{(\ell)}} \| (\bm{d}^{(\ell)})^T \bm{U}^{(\ell)} \|_2^2
		\end{equation}
	\end{enumerate}
 \end{proposition}
 \begin{proof}
 	We follow the proofs of \cite[Theorem 2.1]{hu2013method} and \cite[Proposition 3.4]{boyd2018simplified} in the more general setting of multiplex networks.
 	Starting from \eqref{eq:multiplex_mod_max}, we have
 	\begin{align}
 	Q(\bm{U}) & = \frac{1}{2\mu} \sum_{i,j=1}^n \sum_{k,\ell=1}^L \left[ \left( \bm{A}_{ij}^{(\ell)} - \frac{\gamma^{(\ell)}}{2m^{(\ell)}} \bm{d}_i^{(\ell)}\bm{d}_j^{(\ell)} \right) \delta_{k\ell} + \bm{C}_{j}^{(k\ell)}\delta_{ij} \right] \delta(U(i,k),U(j,\ell))\label{eq:equiv_proof_mod}\\
 	& = \frac{1}{2\mu} \sum_{r=1}^{n_c} \sum_{(i,k)\in S_r,(j,\ell)\in S_r} \left[ \left( \bm{A}_{ij}^{(\ell)} - \gamma^{(\ell)} \frac{\bm{d}_i^{(\ell)}\bm{d}_j^{(\ell)}}{2m^{(\ell)}} \right) \delta_{k\ell} + \bm{C}_{j}^{(k\ell)}\delta_{ij} \right]\nonumber\\
 	& = \frac{1}{2\mu} \sum_{r=1}^{n_c} \sum_{(i,k)\in S_r,(j,\ell)\in S_r} \biggl[ \Bigl( \bm{A}_{(k-1)n+i,(\ell-1)n+j} \delta_{k\ell}\Bigr.\biggr.\nonumber\\
 	& \biggl.\Bigl.\qquad\qquad\qquad\qquad +~\bm{A}_{(k-1)n+j,(\ell-1)n+j}\delta_{ij} \Bigr) - \frac{\gamma^{(\ell)}}{2m^{(\ell)}} \bm{d}_{(\ell-1)n+i}\bm{d}_{(\ell-1)n+j} \delta_{k\ell} \biggr],\nonumber
 	\end{align}
 	where we used that $\bm{A}_{ij}^{(\ell)}\delta_{k\ell} + \bm{C}_j^{(k\ell)}\delta_{ij}$ matches the sparsity structure of the supra-adjacency matrix $\bm{A}$.
 	Next, we rewrite the sum of edge weights between node-layer pairs from the same community as the difference between the sum of edge weights to all node-layer pairs of the network and the sum of edge weights across communities and obtain
 	\begin{align}
	Q(\bm{U}) = & \frac{1}{2\mu} \sum_{r=1}^{n_c} \left[ \sum_{(i,k)\in S_r,(j,\ell)\in \mathcal{V}} \bm{A}_{(k-1)n+i,(\ell-1)n+j} - \sum_{(i,k)\in S_r,(j,\ell)\in S_r^C} \right. \bm{A}_{(k-1)n+i,(\ell-1)n+j}\nonumber\\
 	& \qquad\qquad \left. - \sum_{\ell=1}^L \frac{\gamma^{(\ell)}}{2m^{(\ell)}} \sum_{i,j\in S_r^{(\ell)}} \bm{d}_{(\ell-1)n+i}\bm{d}_{(\ell-1)n+j} \right].\nonumber
 	\end{align}
 	The definition $2\mu = \bm{1}_{nL}^T \bm{A}\bm{1}_{nL}$ applied to the first summand yields
 	\begin{align}
 	& 1 - \frac{1}{2\mu}\sum_{r=1}^{n_c}\sum_{(i,k)\in S_r,(j,\ell)\in S_r^C} \bm{A}_{(k-1)n+i,(\ell-1)n+j}\nonumber\\
 	& \qquad - \frac{1}{2\mu}\sum_{r=1}^{n_c}\sum_{\ell=1}^L \frac{\gamma^{(\ell)}}{2m^{(\ell)}} \sum_{i,j\in S_r^{(\ell)}} \bm{d}_{(\ell-1)n+i}\bm{d}_{(\ell-1)n+j}\nonumber\\
 	=~& 1 - \frac{1}{2\mu}\sum_{r=1}^{n_c}\mathrm{Cut}_{\mathrm{MP}}(S_r,S_r^C) - \frac{1}{2\mu}\sum_{r=1}^{n_c}\sum_{\ell=1}^L \frac{\gamma^{(\ell)}}{2m^{(\ell)}} \sum_{i\in S_r^{(\ell)}} \bm{d}_{(\ell-1)n+i}\mathrm{vol}(S_r^{(\ell)})\label{eq:equiv_proof_cut}\\
 	=~& 1 - \frac{1}{2\mu}|\bm{U}|_{\mathrm{MPTV}} - \frac{1}{2\mu}\sum_{\ell=1}^L \frac{\gamma^{(\ell)}}{2m^{(\ell)}}\sum_{r=1}^c\mathrm{vol}(S_r^{(\ell)})^2\nonumber\\
 	=~& 1 - \frac{1}{2\mu}|\bm{U}|_{\mathrm{MPTV}} - \frac{1}{2\mu}\sum_{\ell=1}^L \frac{\gamma^{(\ell)}}{2m^{(\ell)}}\|(\bm{d}^{(\ell)})^T\bm{U}^{(\ell)}\|_2^2.\label{eq:equiv_proof_btv}
 	\end{align}
 	The assertion follows from \eqref{eq:equiv_proof_mod}, \eqref{eq:equiv_proof_cut}, and \eqref{eq:equiv_proof_btv} since $\frac{1}{2\mu}>0$.
 \end{proof}

 In particular, \Cref{prop:equivalence} shows that the balancing term in the multiplex graph cut and total variation formulation is applied layer-wise.
 Hence, the typically undesired formation of very small communities is counteracted not only on the full multiplex level but on each individual network layer.
 
 As it stands, however, the binary nature of the community affiliation matrix $\bm{U}\in\{0,1\}^{nL\times n_c}$ makes \Cref{prop:equivalence} of little practical use since the equivalent formulations are still of a combinatorial nature.
 Popular approaches to circumvent this issue are resorting to heuristic optimization procedures \cite{blondel2008fast,mucha2010community,traag2019louvain} or relaxed formulations \cite{hu2013method,boyd2018simplified,tudisco2018community,zhang2018sparse,cristofari2020total} of the problem.
 
 As in the single-layer approaches \cite{hu2013method,boyd2018simplified}, we choose to relax $\bm{U}\in\{0,1\}^{nL \times n_c}$ with \mbox{$\bm{U1}_{n_c}=\bm{1}_{nL}$} to a real-valued matrix $\bm{U}\in\R^{nL\times n_c}$ with each of its rows restricted to the Gibbs simplex $\Sigma_{n_c}=\left\{ [\bm{u}_1, \dots, \bm{u}_{n_c}]^T\in[0,1]^{n_c} : \sum_{r=1}^{n_c} \bm{u}_r = 1 \right\}$ and propose a technique to find approximate minimizers of the balanced multiplex total variation functional \eqref{eq:multiplex_BTV}.
 
  \begin{proposition}\label{prop:GL_functional_convergence}
  	The discrete multiplex Ginzburg--Landau functional
  	\begin{equation}\label{eq:GL_functional}
  	\mathcal{F}_{\epsilon,\mathrm{MP}}(\bm{U}) = \mathrm{tr}(\bm{U}^T \bm{L} \bm{U}) + \frac{1}{\epsilon} \sum_{i=1}^n\sum_{\ell=1}^L P(\bm{U}^{(\ell)}_i) + \sum_{\ell=1}^L \frac{\gamma^{(\ell)}}{2m^{(\ell)}} \| (\bm{d}^{(\ell)})^T\bm{U}^{(\ell)} \|_2^2
  	\end{equation}
  	$\Gamma$-converges, i.e., in the limit $\epsilon\rightarrow 0$, to
  	\begin{equation*}
  	|\bm{U}|_{\mathrm{MPTV}} + \sum_{\ell=1}^L \frac{\gamma^{(\ell)}}{2m^{(\ell)}} \| (\bm{d}^{(\ell)})^T \bm{U}^{(\ell)} \|_2^2,
  	\end{equation*}
  	if $\bm{U}$ corresponds to a partition.
  	
  	\begin{sloppypar}
  		Here, $P: \R^{n_c}\rightarrow \R_{\geq 0}$ denotes a multi-well potential with $\frac{P(\bm{U}_i^{(\ell)})}{\|\bm{U}_i^{(\ell)}\|}\rightarrow\infty$ for \mbox{$\|\bm{U}_i^{(\ell)}\|\rightarrow\infty$} and minima in the corners of the Gibbs simplex $\Sigma_{n_c}$, which drives $\bm{U}_i^{(\ell)}\in\R^{n_c}$ towards these minima.
  		The gradient flow of $\mathcal{F}_{\epsilon,\mathrm{MP}}(\bm{U})$ with respect to $\bm{U}$ and the standard Euclidean inner product reads
  	\end{sloppypar}
  	\begin{equation*}
  	\frac{\partial \bm{U}}{\partial t} = -\nabla_{\bm{U}} \mathcal{F}_{\epsilon,\mathrm{MP}}(\bm{U}) = - \left(\bm{L} + \bm{K}\right)\bm{U} - \frac{1}{\epsilon} P'(\bm{U}),
  	\end{equation*}
  	where $\bm{K} = \mathrm{blkdiag}\left[\frac{\gamma^{(1)}}{m^{(1)}}\bm{d}^{(1)}(\bm{d}^{(1)})^T, \dots , \frac{\gamma^{(L)}}{m^{(L)}}\bm{d}^{(L)}(\bm{d}^{(L)})^T\right]$ and $P'$ is applied row-wise.
  \end{proposition}
  \begin{proof}
  	The $\Gamma$-convergence follows from the proof of \cite[Theorem 3.9]{boyd2018simplified} by realizing that the supra-Laplacian $\bm{L}$ defined in \eqref{eq:supra-laplacian} satisfies all properties of unnormalized graph Laplacians of single-layer networks with the particular weighted network structure of the supra-adjacency matrix $\bm{A}$ and that the multi-well potential is defined analogously to the single-layer case.
  	Since $\bm{U}$ can be decomposed into \mbox{$\bm{U}=\begin{bmatrix}(\bm{U}^{(1)})^T & \dots & (\bm{U}^{(L)})^T\end{bmatrix}^T$}, the term $\nabla_{\bm{U}} \sum_{\ell=1}^L \frac{\gamma^{(\ell)}}{2m^{(\ell)}} \| (\bm{d}^{(\ell)})^T\bm{U}^{(\ell)} \|_2^2$ takes the block-diagonal form \mbox{$\bm{K} = \mathrm{blkdiag}\left[\frac{\gamma^{(1)}}{m^{(1)}}\bm{d}^{(1)}(\bm{d}^{(1)})^T, \dots , \frac{\gamma^{(L)}}{m^{(L)}}\bm{d}^{(L)}(\bm{d}^{(L)})^T\right]$}.
  \end{proof}
 
 This relaxation allows the detection of the partitions $\bm{U}$ that approximately minimize balanced multiplex total variation functional, i.e., approximately maximize  multiplex modularity, cf.~\Cref{prop:equivalence}, by computing stationary solutions of the matrix-valued Allen--Cahn equation
 \begin{equation}\label{eq:ODE}
  \frac{\partial \bm{U}}{\partial t} = - \left(\bm{L} + \bm{K}\right)\bm{U} - \frac{1}{\epsilon} P'(\bm{U}).
 \end{equation}
 We discuss the numerical solution of \eqref{eq:ODE} in \Cref{sec:numerical_solution_of_ODEs}.
 
 Similar approaches based on minimizing certain graph Ginzburg--Landau functionals via gradient flows have recently been used in semi-supervised node classification \cite{bertozzi2012diffuse,garcia2014multiclass,budd2021classification,bergermann2021semi}.
 The Ginzburg--Landau functional in that case additionally contains a data fidelity term, which penalizes deviations from community labels that are known \emph{a priori}.
 One possible extension of our method could be to add semi-supervision in a similar manner.
 
 The second major difference between the semi-supervised setting and the modularity maximizing approaches in the single-layer case \cite{hu2013method,boyd2018simplified} as well as the multiplex case \eqref{eq:ODE} is the presence of the (block) rank-$1$ term $\bm{K}$.
 In addition to the usual Laplacian $\bm{L}$ as linear differential operator, $(\bm{L}+\bm{K})$ forms the linear part of the differential equation \eqref{eq:ODE}.
 
 \begin{proposition}\label{prop:spectrum_L_K}
 	Let $\bm{L}\in\R^{nL\times nL}$ be a supra-Laplacian matrix, cf.~\Cref{def:multiplex_network_la}, and $\bm{K}\in\R^{nL\times nL}$ as defined in \Cref{prop:GL_functional_convergence}.
 	Then the discrete linear differential operator $\bm{L} + \bm{K}$ of \eqref{eq:ODE} is symmetric positive semi-definite.
 	If the corresponding multiplex network does not contain a connected component that consists of node-layer pairs that all have zero intra-layer degree, $\bm{L} + \bm{K}$ is strictly positive definite.
 \end{proposition}
 \begin{proof}
 	The symmetry of $\bm{L}$ follows from the assumption of undirectedness and $\bm{K}$ is symmetric by definition.
 	It is well-known that the supra-Laplacian $\bm{L}$ is a positive semi-definite matrix where the multiplicity of the zero eigenvalue corresponds to the number of connected components in the multiplex network \cite{gomez2013diffusion}.
 	For multiplex networks with $m\in\N$ connected components and corresponding subsets of node-layer pairs $C_1,\dots,C_m\subset\mathcal{V}$, the $m$-dimensional null space of $\bm{L}$ is spanned by the eigenvectors $\bm{\chi}_1, \dots,\bm{\chi}_m\in\R^{nL}$, where 
 	\begin{equation*}
 	 [\bm{\chi}_i]_j = \begin{cases}
 	 \frac{1}{\sqrt{|C_i|}}, & \text{if the node-layer pair represented by row $j$ of $\bm{\chi}_i$ belongs to component $C_i$,}\\
 	 0, & \text{otherwise,}
 	 \end{cases}
 	\end{equation*}
 	for $i=1,\dots,m$ and $j=1,\dots,nL$ since $\bm{L}$ can be permuted into block diagonal form with $m$ blocks containing the graph Laplacians of the connected components \cite{von2007tutorial}.
 	
 	By definition, each block $\frac{\gamma^{(\ell)}}{m^{(\ell)}}\bm{d}^{(\ell)}(\bm{d}^{(\ell)})^T, \ell=1,\dots,L$, of $\bm{K}$ has rank $1$ with the only non-zero eigenvalue being $\lambda_1^{(\ell)} = \frac{\gamma^{(\ell)}}{m^{(\ell)}}(\bm{d}^{(\ell)})^T\bm{d}^{(\ell)}>0$ corresponding to the eigenvector $\bm{d}^{(\ell)}$.
 	Thus, $\bm{K}$ is positive semi-definite with rank $L$ and non-zero eigenvalues $\lambda_1^{(1)}, \dots, \lambda_1^{(L)}$ and corresponding eigenvectors $\bm{e}_1\otimes \bm{d}^{(1)}, \dots, \bm{e}_L\otimes \bm{d}^{(L)}$.
 	
 	Writing the smallest eigenvalue of $\bm{L}+\bm{K}$ via the Rayleigh quotient yields
 	\begin{align}\label{eq:rayleigh_L_K}
 	\begin{split}
 	 & \min_{\bm{v}\in\R^{nL}, \|\bm{v}\|_2=1} \bm{v}^T(\bm{L}+\bm{K})\bm{v}\\
 	 = & \min_{\bm{v}\in\R^{nL}, \|\bm{v}\|_2=1} \underbrace{\bm{v}^T\bm{L}\bm{v}}_{\geq 0} + \underbrace{\frac{\gamma^{(1)}}{m^{(1)}}\|(\bm{d}^{(1)})^T\bm{v}^{(1)}\|_2^2}_{\geq 0} + \dots + \underbrace{\frac{\gamma^{(L)}}{m^{(L)}}\|(\bm{d}^{(L)})^T\bm{v}^{(L)}\|_2^2}_{\geq 0},
 	 \end{split}
 	\end{align}
 	with the usual block vector notation
 	$\bm{v} = \begin{bmatrix}
 	(\bm{v}^{(1)})^T & \dots & (\bm{v}^{(L)})^T
 	\end{bmatrix}^T$, which shows the positive semi-definiteness.
 	
 	Starting with the term $\bm{v}^T\bm{L}\bm{v}$ in \eqref{eq:rayleigh_L_K}, only possible candidates for eigenvectors to zero eigenvalues of $\bm{L}+\bm{K}$ are $\bm{\chi}_1, \dots,\bm{\chi}_m\in\R^{nL}$.
 	By definition, $\bm{d}^{(\ell)}$ is entry-wise non-negative with zero entries representing intra-layer-isolated node-layer pairs\footnote{The node-layer pairs may possess inter-layer edges since these do not enter the vectors $\bm{d}^{(1)}, \dots , \bm{d}^{(L)}$.}.
 	Hence, $\bm{L}+\bm{K}$ can have a zero eigenvalue only if the block degree vector $\bm{d} = \begin{bmatrix}(\bm{d}^{(1)})^T & \dots & (\bm{d}^{(L)})^T\end{bmatrix}^T$ contains zeros on the non-zero pattern of at least one of the vectors $\bm{\chi}_1, \dots,\bm{\chi}_m$.
 \end{proof}
 
 \begin{remark}
  With the choice $\widetilde{\bm{A}}=\bm{11}^T - \bm{I}$ of inter-layer couplings that we assume throughout the manuscript, the only possible scenario of semi-definiteness (as opposed to strict definiteness) of $\bm{L}+\bm{K}$ is the occurrence of at least one physical node whose node-layer pairs are isolated within all layers.
 \end{remark}

\section{Direct multiplex modularity maximization}\label{sec:dir}

 We call the second method proposed in this work direct gradient flow multiplex modularity maximization (\dir).
 It builds on the realization that a trivial reformulation of multiplex modularity maximization
 \begin{equation*}
 \max_{\substack{\bm{U}\in\{0,1\}^{nL \times n_c}\\\bm{U1}=\bm{1}}} \frac{1}{2\mu} \mathrm{tr}\left(\bm{U}^T\bm{M}\bm{U}\right)
 \end{equation*}
 into
 \begin{equation}\label{eq:dir_mod_min}
 \min_{\substack{\bm{U}\in\{0,1\}^{nL \times n_c}\\\bm{U1}=\bm{1}}} - \mathrm{tr}\left(\bm{U}^T\bm{M}\bm{U}\right)
 \end{equation}
 takes the form of a graph Dirichlet energy minimization problem with indefinite discrete linear differential operator $-\bm{M}\in\R^{nL\times nL}$.
 In particular, the minimization of such a graph Dirichlet energy with the supra-Laplacian $\bm{L}$ as positive semi-definite discrete linear differential operator appears in \eqref{eq:GL_functional} and constitutes a sub-problem of our first method \texttt{MPBTV}.
 
 Analogously to \Cref{sec:MPBTV}, we relax the binary nature of the community affiliation matrix \mbox{$\bm{U}\in\{0,1\}^{nL\times n_c}$} to $\bm{U}\in\R^{nL\times n_c}$ with its rows restricted to the Gibbs simplex $\Sigma_{n_c}$ and propose to solve the relaxed problem via a gradient flow with respect to $\bm{U}$, i.e.,
 \begin{equation}\label{eq:ODE_dir}
  \frac{\partial\bm{U}}{\partial t} = - \nabla_{\bm{U}} \left( - \mathrm{tr}\left(\bm{U}^T\bm{M}\bm{U}\right)\right) = \bm{M}\bm{U}.
 \end{equation}
 This results in a linear first-order matrix-valued ordinary differential equation (ODE) whose numerical solution is discussed in the next section.

\section{Numerical solution of the ODEs}\label{sec:numerical_solution_of_ODEs}
 
 The two methods \texttt{MPBTV} and \dir~introduced in \Cref{sec:MPBTV,sec:dir} are derived as gradient flows of suitable energy functionals and take the form of the matrix-valued differential equations \eqref{eq:ODE} and \eqref{eq:ODE_dir}.
 In this section, we discuss their numerical solution that is based on approximating the corresponding discrete linear differential operators by a subset of their eigenvalues and \mbox{-vectors} and by integrating the temporal derivative with a graph Merriman--Bence--Osher (MBO) scheme \cite{merriman1994motion,merkurjev2013mbo}.

 Although \eqref{eq:ODE} and \eqref{eq:ODE_dir} describe the temporal evolution of community affiliation functions defined on the node-layer pairs of multiplex networks, their numerical solution is to be understood as an optimization procedure minimizing the corresponding energy functional without a physical interpretation of time.
 It can not be guaranteed that the relaxed solutions obtained by our gradient flows correspond to global optima of either objective and it is well-known that the modularity energy landscape typically possesses multiple local optima \cite{good2010performance,boyd2018simplified}.
 We hence aim at exploring different parts of the energy landscape by computing stationary solutions of the ODEs for several, say $n_{\mathrm{runs}}\in\N$, initial conditions $\bm{U}_0\in\{0,1\}^{nL\times n_c}$ for which we randomly draw one standard basis vector from $\R^{n_c}$ per row.
 Experiments with real-valued random initial conditions $\bm{U}_0\in\R^{nL\times n_c}$ led to similar classification results and we choose binary initial conditions to more closely match the original combinatorial problem formulation.
 As it is customary for classification techniques relying on randomized initial conditions and since we have a quality function for comparing different solutions at hand, we then choose the solution obtaining the highest multiplex modularity score as community partition.
 
 We now discuss our procedure for solving the initial value problems \eqref{eq:ODE} and \eqref{eq:ODE_dir} together with the initial condition $\bm{U}_0\in\{0,1\}^{nL\times n_c}$.
 The linear part of \eqref{eq:ODE} is solved exactly by $e^{-t(\bm{L} + \bm{K})}\bm{U}_0$ for all $t>0$, where $e^{-t(\bm{L} + \bm{K})}$ denotes the matrix exponential of the negative \mbox{(semi-)} definite discrete linear differential operator $-(\bm{L} + \bm{K})\in\R^{nL\times nL}$, cf.~\Cref{prop:spectrum_L_K}.
 Similarly, the exact solution of \eqref{eq:ODE_dir} reads $e^{t\bm{M}}\bm{U}_0$ for all $t>0$ and the indefinite multiplex modularity matrix $\bm{M}\in\R^{nL\times nL}$.
 
 Due to the undirectedness of the considered multiplex networks, both $-(\bm{L} + \bm{K})$ and $\bm{M}$ are symmetric and hence diagonalizable and denoting the eigenvalues and eigenvectors normalized in Euclidean norm of either matrix by $\lambda_i\in\R$ and $\bm{\phi}_i\in\R^{nL}, i=1,\dots,nL$, respectively, their matrix exponentials can be written as
 \begin{equation}\label{eq:matexp_spectral}
  \sum_{i=1}^{nL} e^{\lambda_i} \bm{\phi}_i\bm{\phi}_i^T,
 \end{equation}
 where $e^{\lambda_i}\in\R$ denotes the scalar exponential function applied to the eigenvalues \cite{higham2008functions}.
 Since the eigenvectors are normalized, the magnitude of $e^{\lambda_i}$ determines the relative contribution of the terms $e^{\lambda_i}\bm{\phi}_i\bm{\phi}_i^T, i=1,\dots,nL$ to the sum \eqref{eq:matexp_spectral}.
 Hence, in our case, $e^{-t(\bm{L} + \bm{K})}$ is dominated by the small eigenvalues of $(\bm{L} + \bm{K})$ while $e^{t\bm{M}}$ is dominated by the large positive eigenvalues of $\bm{M}$.
 
 Clearly, the truncation of \eqref{eq:matexp_spectral} to a relatively small number of $k\in\N$ eigenvalues and -vectors tremendously increases the efficiency of evaluating the corresponding matrix exponential\footnote{Note that alternative truncation strategies are discussed in \cite{budd2021classification}, which may be more appropriate for very small time step sizes $\Delta t$.}.
 We further found $k$ to be the most decisive hyper-parameter of our methods in terms of influencing the quality of the obtained community detection results.
 This is not entirely surprising since similar observations have been made for node classification methods relying on the same spectral information of similar matrices.
 It is well-known that the eigenvectors corresponding to the smallest non-zero eigenvalues of different graph Laplacians solve several balanced multiclass graph cut problems \cite{von2007tutorial}.
 In particular, the signs of the entries of the Fiedler vector yield optimal bi-partitions of single-layer networks in this sense.
 Moreover, graph semi-supervised node classification techniques based on solving fidelity-forced Allen--Cahn equations typically rely on approximations of the chosen graph Laplacian by its small eigenvalues and corresponding eigenvectors \cite{bertozzi2012diffuse,garcia2014multiclass,budd2021classification,bergermann2021semi}.
 Similarly, the sign of the eigenvector to the largest positive eigenvalue of the single-layer modularity matrix, cf.~\Cref{rem:single-layer_modularity}, constitutes Newman's spectral method for maximizing modularity for network bi-partitions \cite{newman2006modularity,newman2006finding}.
 Hence, in this sense, our method \dir\ provides the means to utilize additional spectral information of the multiplex modularity matrix for finding modularity-maximizing partitions of node-layer pairs into an arbitrary number of communities.
 
 For the approximation of the desired eigenvalues and -vectors we rely on the Krylov--Schur method as implemented in \texttt{Matlab}'s \texttt{eigs} function that iteratively extends a Krylov subspace, transforms it to Schur form, and discards non-converged eigenvalues \cite{stewart2002krylov}.
 The computational bottleneck of this method is posed by the matrix-vector products underlying the generation of the Krylov subspaces.
 This technique is computationally efficient and scalable to large-scale problems since the supra-adjacency and -Laplacian of multiplex networks are sparse with $\mathcal{O}(nL)$ non-zero entries in many applications\footnote{Note that for fully-connected kernel adjacency matrices Fourier methods can be leveraged to obtain similar linear runtime complexities, cf., e.g., \cite{bergermann2021semi}.}.
 Moreover, matrix-vector multiplications of $\bm{K}$ with \mbox{$\bm{v} = \begin{bmatrix}
 	(\bm{v}^{(1)})^T & \dots & (\bm{v}^{(L)})^T
 	\end{bmatrix}^T$} can be implemented as \texttt{Matlab} function handles via inner products $(\bm{d}^{(\ell)})^T\bm{v}^{(\ell)}, \ell=1,\dots,L$ without explicitly forming the dense diagonal blocks of $\bm{K}$.
 In practice, we observed that the computation of the large eigenvalues and -vectors of the multiplex modularity matrix $\bm{M}$ was usually faster than the computation of the small eigenvalues and -vectors of $\bm{L}+\bm{K}$ since small-magnitude eigenvalue computations typically require more expensive shift-and-invert strategies \cite{golub2013matrix}, which is implemented in \texttt{Matlab}'s \texttt{eigs} function.
 Moreover, for some multiplex networks we observed somewhat clustered small eigenvalues in $\bm{L}+\bm{K}$, which typically slows down the solver's speed of convergence, cf., e.g., \cite{golub2013matrix}.
 We report timings of  (offline) eigenvalue computations in \Cref{tab:GT_networks_runtimes} and \Cref{fig:runtimes_beach_image_offline}.
 
 So far in this section, we have ignored the non-linear derivative $P'(\bm{U})\in\R^{nL\times n_c}$ of the multi-well potential in \eqref{eq:ODE}.
 For certain IMEX schemes that are widely used for the time integration of similar problems, the evaluation of $P'(\bm{U})$ in each time step is the computationally most expensive part after the discrete linear differential operator has been approximated by a subset of its eigevalues and -vectors \cite{bertozzi2012diffuse,garcia2014multiclass,budd2021classification,bergermann2021semi}.
 An alternative approach that is also heavily used for the considered class of non-linear graph-based ODEs leverages a graph extension of the continuum Merriman--Bence--Osher (MBO) scheme \cite{merriman1994motion} that approximates the mean curvature flow of non-linear reaction-diffusion equations via thresholding dynamics \cite{merkurjev2013mbo}.
 The graph MBO scheme consists of two iteratively repeated stages: a diffusion step integrating the linear part of the ODE and a thresholding step that maps each node's community affiliation vector to the closest standard basis vector, i.e., the community it most likely belongs to.
 
 The MBO diffusion step of \eqref{eq:ODE} and one time step of \eqref{eq:ODE_dir} are realized by evaluating the truncated version of \eqref{eq:matexp_spectral} with $k$ summands as discussed previously for a time step $\Delta t>0$.
 The thresholding step reflects the contribution of the multi-well potential $\sum_{i=1}^n\sum_{\ell=1}^L P(\bm{U}^{(\ell)}_i)$ in \eqref{eq:GL_functional} to drive the rows of $\bm{U}$ towards corners of the simplex $\Sigma_{n_c}$.
 In the single-layer case it has been shown that the graph MBO scheme is one member of a class of semi-discrete implicit Euler schemes solving fidelity-forced graph Allen--Cahn equations \cite{budd2021classification}.
 Furthermore, replacing the rows of $\bm{U}$ by standard basis vectors can be interpreted as relating the relaxed formulation $\bm{U}\in\R^{nL\times n_c}$ of \cref{prop:GL_functional_convergence} to the binary formulation $\bm{U}\in\{0,1\}^{nL\times n_c}$ of \Cref{def:multiplex_modularity} and \Cref{prop:equivalence}.
 
 \begin{algorithm}[t]
 	\vspace{0.5em}
 	\begin{tabular}{lll}
 		\vspace{1mm}
 		Input:
 		& $\bm{\mathcal{D}}$,& Differential operator, $-(\bm{L}+\bm{K})$ for \texttt{MPBTV}, $\bm{M}$ for \dir\\
 		Parameters: & \multicolumn{2}{l}{$\Delta t,\text{tol}>0,$}\\
 		& \multicolumn{2}{l}{$n_c,k,n_{\mathrm{runs}},\text{max\_iter}\in\mathbb{N}.$}\\
 	\end{tabular}\\
 	
 	\begin{algorithmic}[1]
 		\State Compute $k$ largest real eigenvalues $\bm{\Lambda}_k$ and corresponding eigenvectors $\bm{\Phi}_k$ of $\bm{\mathcal{D}}$
 		
 		\For{$i=1:n_{\mathrm{runs}}$}
 		\State Sample random initial conditions $\bm{U}_0\in\{0,1\}^{nL\times n_c}$
 		\For{$j=1:\text{max\_iter}$}
 		\State $\bm{V}_{j} = \bm{\Phi}_k\exp((\Delta t) \bm{\Lambda}_k)\bm{\Phi}_k^T\bm{U}_{j-1}$\hfill\% Diffusion
 		\For{$k=1:nL$}\hfill\% Thresholding
 		\State $k^\ast = \text{argmax}~\bm{V}_{j}(k,:)$
 		\State $\bm{U}_{j}(k,:)=\bm{e}_{k^\ast}^T$
 		\EndFor
 		\If{$\|\bm{U}_{j}-\bm{U}_{j-1}\|_2<\text{tol}$}
 		\State \textbf{break}
 		\EndIf
 		\EndFor
 		\State Compute $Q(\bm{U}_{j})$
 		\EndFor
 	\end{algorithmic}
 	\vspace{0.5em}
 	\begin{tabular}{ll}
 		\vspace{1mm}
 		Output: & $\bm{U}_{j}$ with highest multiplex modularity $Q$.\\
 	\end{tabular}
 	\caption{MBO scheme solving the gradient flow-based multiplex modularity maximization problems \eqref{eq:ODE} and \eqref{eq:ODE_dir}.}\label{alg}
 \end{algorithm}
 
 Since the ODE \eqref{eq:ODE_dir} is linear, there is no theoretical requirement of a thresholding procedure.
 It has, however, been observed for matrix-valued ODEs of this type \cite{garcia2014multiclass,bergermann2021semi} that the diffusion operator may drive rows of the solution $\bm{U}$ far away from the simplex $\Sigma_{n_c}$.
 For this reason, it has become customary to apply a simplex projection after each time step.
 One example technique detects the closest point on $\Sigma_{n_c}$ that does generally not coincide with a standard basis vector as described in \cite{chen2011projection}.
 We have, however, observed experimentally that \eqref{eq:ODE_dir} produces better solutions in terms of multiplex modularity and classification accuracy when MBO thresholding instead of the simplex projection is applied --- presumably due to the relation to the original binary problem formulation discussed earlier.
 Formally, this procedure corresponds to solving the gradient flow of the energy functional \eqref{eq:dir_mod_min} with added multi-well potential term $\sum_{i=1}^n\sum_{\ell=1}^L P(\bm{U}^{(\ell)}_i)$, which drives the solution towards the corners of the simplex.
 
 \Cref{alg} summarizes our procedure for the numerical solution of \eqref{eq:ODE} and \eqref{eq:ODE_dir}.
 After the offline computation of the eigenvalues and -vectors of the relevant discrete linear differential operator in line $1$, the outer for-loop represents the solution of the ODE for $n_{\mathrm{runs}}\in\N$ different random initial conditions $\bm{U}_0$.
 In line $5$, we use the compact notation $\bm{\Lambda}_k=\text{diag}[\lambda_1,\dots,\lambda_k]\in\R^{k\times k}$ and $\bm{\Phi}_k=\begin{bmatrix}
 \bm{\phi}_1 & \dots & \bm{\phi}_k
 \end{bmatrix}\in\R^{nL\times k}$ to represent the truncated version of \eqref{eq:matexp_spectral} as the diffusion step.
 Note that the eigenvalues $\bm{\Lambda}_k$ of $-(\bm{L}+\bm{K})$ are non-positive, cf.~\Cref{prop:spectrum_L_K}, and both matrices $-(\bm{L}+\bm{K})$ and $\bm{M}$ depend on the parameters $\gamma, \omega\in\R_{\geq 0}$.
 The node-layer-wise MBO thresholding in lines $6$ to $9$ can be implemented in an efficient way that avoids the for-loop by determining the row-wise $\arg\max$ of $\bm{V}_j$ and assembling the rows of $\bm{U}_j$ as the corresponding standard basis vectors in a vectorized way.
 As stopping criterion, we prescribe a maximum number of iterations $\text{max\_iter}\in\N$.
 Furthermore, we stop the iteration if the norm of the distance between two consecutive iterates is below a prescribed tolerance $\text{tol}\ll 1$ indicating that the trajectory has reached or come close to a stationary point.

\section{Numerical experiments}\label{sec:numerical_experiments}
 
 In this section, we test our methods \texttt{MPBTV} and \dir~on a range of real-world multiplex networks of different types and sizes.
 We further consider five community detection methods from the literature that yield non-overlapping partitions of the node-layer pairs of multiplex networks: GenLouvain \cite{mucha2010community}, Leiden MP \cite{traag2019louvain}, Infomap \cite{de2015identifying}, MDLP \cite{boutemine2017mining}, and LART \cite{kuncheva2015community}\footnote{We use the \texttt{Matlab} implementation of GenLouvain that is publicly available under \url{https://github.com/GenLouvain/GenLouvain} and contains a full matrix version as well as a function handle version, the multiplex functionality of the \texttt{python} package leidenalg, cf.~\url{https://leidenalg.readthedocs.io/en/stable/multiplex.html}, the \texttt{python} implementation of Infomap and MDLP from the \texttt{multinet} library, cf.~\url{https://github.com/uuinfolab/py_multinet/}, and the \texttt{python} implementation of LART that is available under \url{https://bitbucket.org/uuinfolab/20csur/src/master/algorithms/LART/}.}.
 GenLouvain and Leiden MP are designed to optimize modularity, Infomap and LART are random walk methods, and MDLP is based on label propagation.
 We compare the performance of the methods with respect to multiplex modularity and if ground truth labels are available we additionally compare classification accuracies and NMI \cite{strehl2002cluster} scores\footnote{We define classification accuracy as the proportion of correctly classified node-layer pairs with respect to ground truth labels.
 The unsupervised nature of the methods requires the matching of detected communities to ground truth communities before computing classification accuracies and NMIs.
 To this end, we loop over the detected communities in descending order with respect to their cardinality and match each to the respective ground truth community with the largest overlap (without replacement, e.g., no ground truth community is assigned to more than one detected community).}.
 We ran LART both with the (recommended) standard resolution parameter $\gamma=1$ and the value of $\gamma$ used in our methods as well as GenLouvain and report the LART results obtaining the higher multiplex modularity.
 For methods depending on random initializations, i.e., different random conditions $\bm{U}_0$ for our methods or randomized node-layer pair orderings in GenLouvain, we report maximum values over $n_{\mathrm{runs}}\in\N$ initializations.
 In \Cref{tab:small_networks_modularity,tab:GT_networks_modularity,tab:GT_networks_accuracy,tab:GT_networks_NMI,tab:beach_64th_comparison}, dark gray boxes with bold font highlight largest and light gray boxes second-largest values of the displayed quantities.
 
 \texttt{Matlab} codes for replicating the numerical experiments are publicly available under \url{https://github.com/KBergermann/GradFlowModMax}.
 All runtimes are measured on an Intel i5-8265U CPU with $4 \times 1.60-3.90$ GHz cores, 16 GB RAM, and \texttt{Matlab} R2020b.
 Termination criteria of \Cref{alg} are chosen $\text{max\_iter} = 300$ and $\text{tol} = 10^{-8}$.

\subsection{Small real-world multiplex networks}
 
 We start by running the multiplex community detection methods on five small, unweighted, and publicly available\footnote{under \url{https://manliodedomenico.com/data.php}} real-world multiplex networks from genetic and social applications.
 \Cref{tab:small_networks_hyper_parameters} shows the number $n$ of nodes and $L$ of layers of each network as well as hyper-parameter values for $\gamma, n_c,$ and $k$ that were found to yield high multiplex modularity scores for \texttt{MPBTV} and \dir~in a random search.
 The remaining hyper-parameters are chosen $\omega = 1, \Delta t = 1,$ and $n_{\mathrm{runs}}=50$ throughout this subsection.
 
 \Cref{tab:small_networks_modularity} shows the maximum multiplex modularity scores obtained by the different methods over $n_{\mathrm{runs}}=50$ random initializations.
 Infomap did not produce valid community partitions\footnote{i.e., a subset of node-layer pairs was not assigned a community label} on the ``Danio rerio'' and ``Human herpes 4'' networks.
 Our two methods either obtain the highest multiplex modularities or range $0.001$ or $0.002$ below the highest score.
 \texttt{MPBTV} and \dir~obtain the improvement of multiplex modularity over GenLouvain for the Florentine families network by assigning family ``Barbadori'' to the ``Medici'' community while keeping the overall number of communities constant.
 
 \begin{table}[b]
 	\centering
 	\begin{tabular}{lccccc}
 		\hline\hline
 		&Danio r.&Florent.\ fam.&Hepatitus C&H.\ herpes 4&Oryctolagus\\\hline\hline
 		$n$ &156&17&106&217&145\\
 		$L$ &5&2&3&4&3\\\hline
 		$\gamma$ &1.2&0.6&1.5&1&0.4\\
 		$n_c$ &16&3&40&11&13\\
 		$k$ (\texttt{MPBTV}) &30&4&90&13&18\\
 		$k$ (\dir) &27&7&80&12&18\\\hline\hline
 	\end{tabular}
 	\caption{Network size of five small real-world multiplex networks and hyper-parameter values used in the numerical experiments.
 		``Danio r.'' abbreviates ``Danio rerio'', ``Florent.\ fam.'' stands for ``Florentine families'', and ``H.\ herpes 4'' for ``Human herpes 4''.}\label{tab:small_networks_hyper_parameters}
 \end{table}
 
 \begin{table}
 	\centering
 	\begin{tabular}{lccccc}
 		\hline\hline
 		&Danio r.&Florent.\ fam.&Hepatitus C&H.\ herpes 4&Oryctolagus\\\hline\hline
 		\texttt{MPBTV} & \second{0.974} & \first{0.681} & \second{0.668} & 0.906
 		& \second{0.949}\\
 		\dir & 0.965 & \first{0.681} & \first{0.682} & \second{0.909} & 0.945\\\hline
 		GenLouvain \cite{mucha2010community} & \first{0.976} & \second{0.673} & 0.627 & \first{0.910} & \first{0.950}\\
 		Leiden MP \cite{traag2019louvain}& \first{0.976} & \first{0.681} & 0.627 & \first{0.910} & \first{0.950}\\
 		Infomap \cite{de2015identifying} & --- & 0.659 & 0.583 & --- & \first{0.950}\\
 		LART \cite{kuncheva2015community} & 0.084 & 0.558 & -0.032 & 0.095 & 0.201\\
 		MDLP \cite{boutemine2017mining} & 0.969 & 0.633 & 0.583 & 0.908 & \first{0.950}\\\hline\hline
 	\end{tabular}
 	\caption{Multiplex modularity scores obtained by different community detection methods in the setting of \Cref{tab:small_networks_hyper_parameters}.}\label{tab:small_networks_modularity}
 \end{table}

\subsection{Ground truth multiplex networks}\label{sec:numerical_experiments_GT_networks}
 
 In the second set of experiments, we consider the seven data sets 3sources~\cite{liu2013multi}, BBC~\cite{greene2005producing}, BBCS~\cite{greene2009matrix}, Citeseer~\cite{lu2003link}, Cora~\cite{mccallum2000automating}, WebKB~\cite{craven1998learning}, and Wikipedia~\cite{rasiwasia2010new} containing feature vector data from different data sources as well as ground truth labels for all data points.
 We construct small- to medium-sized multiplex networks by identifying each data source with a layer and each data point with a physical node.
 Analogously to \cite{mercado2018power}, we construct sparse weighted nearest-neighbor graphs for each layer where distances are measured and edges are weighted by the feature vectors' Pearson correlations.
 
 \begin{table}
 	\centering
 	\begin{tabular}{lccccccc}
 		\hline\hline
 		&3Sources&BBC&BBCS&Citeseer&Cora&WebKB&Wiki.\\\hline\hline
 		$n$ & 169 & 685 & 544 & 3312 & 2708 & 187 & 693\\
 		$L$ & 3 & 4 & 2 & 2 & 2 & 2 & 2\\\hline
 		$\gamma$ & 1.2 & 0.8 & 0.6 & 0.6 & 0.8 & 0.6 & 1\\
 		$n_c$ & 6 & 5 & 5 & 6 & 7 & 5 & 10\\
 		$k$ (\texttt{MPBTV}) & 10 & 4 & 34 & 20 & 18 & 3 & 13\\
 		$k$ (\dir) & 12 & 11 & 11 & 27 & 28 & 6 & 18\\\hline\hline
 	\end{tabular}
 	\caption{Network size of seven medium-sized multiplex networks with ground truth labels.
 		The numbers of communities $n_c$ are prescribed by the ground truth, the resolution parameters $\gamma$ are chosen such that GenLouvain detects $n_c$ communities, and the numbers $k$ of eigenvalues and -vectors for \texttt{MPBTV} and \dir~are chosen by a grid search.
 		``Wiki.'' abbreviates ``Wikipedia''.}\label{tab:GT_networks_hyper_parameters}
 \end{table}
 
 \Cref{tab:GT_networks_hyper_parameters} shows the number $n$ of nodes and $L$ of layers of the resulting multiplex networks.
 Since the number $n_c$ of communities is known \emph{a priori}, we performed a grid search over the resolution parameter $\gamma$ and choose the value for which GenLouvain correctly identifies the prescribed number $n_c$ of communities.
 The numbers $k$ of eigenvalues and -vectors used by \texttt{MPBTV} and \dir~were also chosen by a grid search.
 \Cref{tab:GT_networks_runtimes} shows that this can be performed quite inexpensively once the largest considered number $k$ of eigenvalues and -vectors has been obtained in offline computations.
 The table compares runtimes of the \texttt{Matlab} implementations of \texttt{MPBTV}, \dir, and GenLouvain\footnote{The results in \Cref{tab:GT_networks_runtimes} use the GenLouvain implementation that is based on assembling the full multiplex modularity matrix.
 	An alternative implementation with function handles is considered in \Cref{fig:runtimes_beach_image}.}.
 Runtimes for other methods are not included since these are implemented in \texttt{python}.
 
 \Cref{tab:GT_networks_runtimes} shows that for most networks the runtime of offline computations of the largest real eigenvalues and corresponding eigenvectors of $-(\bm{L}+\bm{K})$ for \texttt{MPBTV} and $\bm{M}$ for \dir~is of the same order as one run of GenLouvain, i.e., one execution with randomized node-layer pair ordering.
 One run of \texttt{MPBTV} and \dir~corresponds to the solution of \eqref{eq:ODE} and \eqref{eq:ODE_dir} by \Cref{alg} for one initial condition $\bm{U}_0$, respectively, which is orders of magnitude faster.
 Moreover, the per run runtime of \texttt{MPBTV} and \dir~scales better with respect to the problem size than that of GenLouvain, which leads to smaller factors between the runtimes for the smaller 3Sources and WebKB networks.
 The ``Citeseer'' and ``Cora'' networks are examples in which the discrete linear differential operators $\bm{L}+\bm{K}$ exhibit clustered small eigenvalues, cf.~\Cref{sec:numerical_solution_of_ODEs}.
 The relatively long offline runtimes of \texttt{MPBTV} for these networks are owed to an increase of the Krylov subspace dimension in the Krylov--Schur eigensolver from the standard value $2k$ to $4k$.
 The remaining hyper-parameters are chosen $\omega=1, \Delta t=0.4, n_{\mathrm{runs}}=20$ throughout this subsection.
 
 \begin{table}
 	\centering
 	\begin{tabular}{llccccccc}
 		\hline\hline
 		&&3Sources&BBC&BBCS&Citeseer&Cora&WebKB&Wiki.\\\hline\hline
 		\multirow{2}{*}{\texttt{MPBTV}} & offline & 0.033 & 0.218 & 0.220 & 37.866 & 29.939 & 0.032 & 0.174\\
 		& per run & 0.001 & 0.002 & 0.002 & 0.011 & 0.010 & 0.001 & 0.005\\\hline
 		\multirow{2}{*}{\dir} & offline & 0.012 & 0.163 & 0.035 & 3.708 & 2.678 & 0.006 & 0.059\\
 		& per run & 0.001 & 0.004 & 0.002 & 0.015 & 0.011 & 0.002 & 0.003\\\hline
 		\multirow{2}{*}{GenLouvain \cite{mucha2010community}} & offline & 0.003 & 0.074 & 0.022 & 2.700 & 1.761 & 0.004 & 0.035\\
 		& per run & 0.040 & 0.442 & 0.223 & 11.071 & 5.988 & 0.037 & 0.503\\\hline\hline
 	\end{tabular}
 	\caption{Runtimes of the three methods implemented in \texttt{Matlab} in seconds.
 		The offline computations of GenLouvain denote the assembling of the multiplex modularity matrix while ``per run'' refers to the execution of the method for one randomized node-layer pair ordering.
 		Offline runtimes of \texttt{MPBTV} and \dir~correspond to eigenvalue computations while ``per run'' refers to the solution of \eqref{eq:ODE_dir} and \eqref{eq:ODE}, respectively, cf.~\Cref{sec:numerical_solution_of_ODEs}.
 		All runtimes are averaged over $10$ independent code executions.}\label{tab:GT_networks_runtimes}
 \end{table}
 
 \Cref{tab:GT_networks_accuracy,tab:GT_networks_modularity,tab:GT_networks_NMI} compare the performance of all multiplex community detection methods in terms of multiplex modularity, classification accuracy, and NMI.
 All values denote maximum values over $n_{\mathrm{runs}}=20$ runs.
 \Cref{tab:GT_networks_modularity} shows that GenLouvain obtains the highest multiplex modularity scores for essentially all networks with Leiden MP, \texttt{MPBTV}, and \dir~often trailing closely behind.
 Note that for ``WebKB'', the methods Infomap, LART, and MDLP obtain higher multiplex modularity scores with the trivial classifier placing all node-layer pairs in the same community.
 This example shows that modularity is not necessarily ideal as sole quality measure since distinctly higher classification accuracies and NMIs are obtained by non-trivial partitions detected by \texttt{MPBTV}, \dir, GenLouvain, and Leiden MP, cf.~\Cref{tab:GT_networks_accuracy,tab:GT_networks_NMI}.
 Overall, the results in \Cref{tab:GT_networks_accuracy,tab:GT_networks_NMI} show that \texttt{MPBTV} and \dir~compare favorably with the remaining methods with respect to classification accuracy and NMI.
 In particular, our methods often beat GenLouvain in these metrics albeit obtaining slightly lower multiplex modularity scores.

\begin{table}
	\centering
	\begin{tabular}{lccccccc}
		\hline\hline
		&3Sources&BBC&BBCS&Citeseer&Cora&WebKB&Wiki.\\\hline\hline
		\texttt{MPBTV} &0.581&\second{0.721}&0.506&0.608&0.511&0.351&\second{0.608}\\
		\dir &\second{0.584}&0.714&0.505&0.591&0.486&0.465&0.442\\\hline
		GenLouvain \cite{mucha2010community} &\first{0.593}&\first{0.742}&\first{0.511}&\first{0.617}&\first{0.559}&\second{0.479}&\first{0.691}\\
		Leiden MP \cite{traag2019louvain}& 0.579 & 0.716 & \second{0.508} & \second{0.612} & \second{0.553} & \first{0.607} & \first{0.691}\\
		Infomap \cite{de2015identifying} & 0.533 & 0.717 & 0.444 & 0.554 & 0.502 & \first{0.607} & 0.513\\
		LART \cite{kuncheva2015community} & 0.179 & 0.529 & 0.444 & 0.485 & 0.286 & \first{0.607} & 0.046\\
		MDLP \cite{boutemine2017mining} & 0.582 & 0.581 & 0.501 & 0.485 & 0.286 & \first{0.607} & 0.424\\\hline\hline
	\end{tabular}
	\caption{Multiplex modularity scores obtained by different community detection methods in the setting of \Cref{tab:GT_networks_hyper_parameters}.}\label{tab:GT_networks_modularity}
\end{table}

\begin{table}
	\centering
	\begin{tabular}{lccccccc}
		\hline\hline
		&3Sources&BBC&BBCS&Citeseer&Cora&WebKB&Wiki.\\\hline\hline
		\texttt{MPBTV} &\second{0.876}&\first{0.910}&\first{0.911}&\first{0.722}&0.544&0.521&0.398\\
		\dir &0.838&\second{0.877}&0.808&0.648&0.454&\second{0.666}&0.478\\\hline
		GenLouvain \cite{mucha2010community} &0.852&0.854&0.841&\second{0.712}&\first{0.692}&0.663&0.413\\
		Leiden MP \cite{traag2019louvain}& \first{0.913} & 0.830 & 0.623 & 0.598 &\second{0.612} & \first{0.694} & 0.403\\
		Infomap \cite{de2015identifying} & 0.716 & 0.685 & 0.355 & 0.477 & 0.281 & 0.551 & \first{0.592}\\
		LART \cite{kuncheva2015community} & 0.424 & 0.330 & 0.355 & 0.212 & 0.302 & 0.551 & 0.150\\
		MDLP \cite{boutemine2017mining} & 0.834 & 0.247 & \second{0.906} & 0.212 & 0.302 & 0.551 & \second{0.489}\\\hline\hline
	\end{tabular}
	\caption{Classification accuracies obtained by different community detection methods in the setting of \Cref{tab:GT_networks_hyper_parameters}.}\label{tab:GT_networks_accuracy}
\end{table}

\begin{table}
	\centering
	\begin{tabular}{lccccccc}
		\hline\hline
		&3Sources&BBC&BBCS&Citeseer&Cora&WebKB&Wiki.\\\hline\hline
		\texttt{MPBTV} & \second{0.804} & \first{0.768} & 0.762 & \first{0.452} & 0.388 & 0.238 & 0.303\\
		\dir & \first{0.813} & 0.694 & 0.678 & \second{0.372} & 0.320 & \first{0.340} & \second{0.512}\\\hline
		GenLouvain \cite{mucha2010community} & 0.800 & \second{0.710} & \second{0.777} & \first{0.452} & \first{0.519} & 0.243 & 0.249\\
		Leiden MP \cite{traag2019louvain}& 0.787 & 0.645 & 0.588 & 0.364 & \second{0.409} & \second{0.339} & 0.242 \\
		Infomap \cite{de2015identifying} & 0.607 & 0.618 & 0 & 0.329 & 0.334 & 0 & \first{0.544}\\
		LART \cite{kuncheva2015community} & 0.428 & 0.085 & 0 & 0 & 0 & 0 & 0.480\\
		MDLP \cite{boutemine2017mining} & 0.731 & 0.353 & \first{0.826} & 0 & 0 & 0 & 0.456\\\hline\hline
	\end{tabular}
	\caption{Normalized mutual information (NMI) scores obtained by different community detection methods in the setting of \Cref{tab:GT_networks_hyper_parameters}.}\label{tab:GT_networks_NMI}
\end{table}

\subsection{Image data}
 
 The final set of numerical experiments is performed on the labeled image shown in \Cref{fig:beach_image} created by the authors.
 The same image has previously been studied in semi-supervised node classification \cite{bergermann2021semi}. 
 We construct multiplex networks with $L=2$ layers that separately record RGB channels and $xy$ pixel coordinates.
 Identifying each pixel with a physical node, this separate treatment allows the construction of unweighted nearest-neighbor graphs with $40$ neighbors in the RGB layer and $10$ neighbors in the $xy$ layer where nearest neighbors are determined by Pearson correlations of the corresponding feature vectors.
 We experimentally found this model to yield superior classification results compared to single-layer approaches taking only the RGB channels or RGB and $xy$ information as a $5$-dimensional feature space into account.
 
 \Cref{fig:beach_image_GT} shows that ground truth labels consist of $n_c=5$ classes.
 However, since the class ``swimming object'' occupies less than $1\%$ of all pixels and is difficult to detect due to heterogeneous distributions in both RGB and $xy$ features, we found the choice $n_c=4$ to give better classification results across different community detection methods.
 We choose a relatively large layer-coupling parameter $\omega = 10$ since smaller values induce an overly high degree of smoothness in the labels of the $xy$ layer that introduces classification errors.
 Due to the relatively high dimension of the modularity landscape ($nL=304\,720$ for the multiplex network corresponding to \Cref{fig:beach_image_im}) we further choose a relatively large number $n_{\mathrm{runs}}=100$ of random initial conditions for all methods.
 We utilized the available ground truth information to find the resolution parameter $\gamma=0.1$ that is required for GenLouvain to correctly identify the prescribed number $n_c$ of communities.
 Exceptions are the downscaled version\footnote{In the context of image data, we refer to downscaling when the number of pixels in both $x$ and $y$ direction is reduced by a common factor of $2$ without smoothing.} of \Cref{fig:beach_image_im} considered in \Cref{tab:beach_64th_comparison} as well as the three smallest networks in \Cref{fig:runtimes_beach_image}, which require the choice $\gamma=0.2$.
 The remaining hyper-parameters are chosen $k=9$ and $\Delta t = 1$ for both \texttt{MPBTV} and \dir~throughout this subsection.
 Similarly to \Cref{sec:numerical_experiments_GT_networks}, we increase the Krylov subspace dimension in the Krylov--Schur eigensolver for \texttt{MPBTV} from the standard value $2k$ to $3k$ to ensure convergent eigencomputations.
 
 \begin{figure}
 	\subfloat[Image]{
 		\includegraphics[width=.49\textwidth]{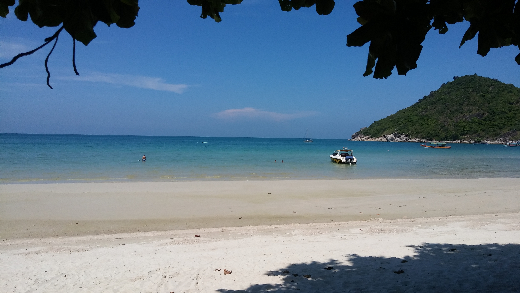}\label{fig:beach_image_im}
 	}
 	\subfloat[Ground truth labels]{
 		\includegraphics[width=.49\textwidth]{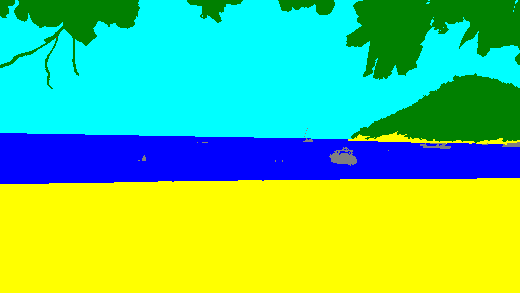}\label{fig:beach_image_GT}
 	}
 	\caption{Test image with $293 \times 520$ pixels and corresponding ground truth labels.
 		Green label colors represent ``tree'', yellow ``beach'', dark blue ``sea'', light blue ``sky'', and gray ``swimming object'' such as boats or humans.}\label{fig:beach_image}
 \end{figure}
 
 \begin{figure}
 	\subfloat[RGB layer]{
 		\includegraphics[width=.49\textwidth]{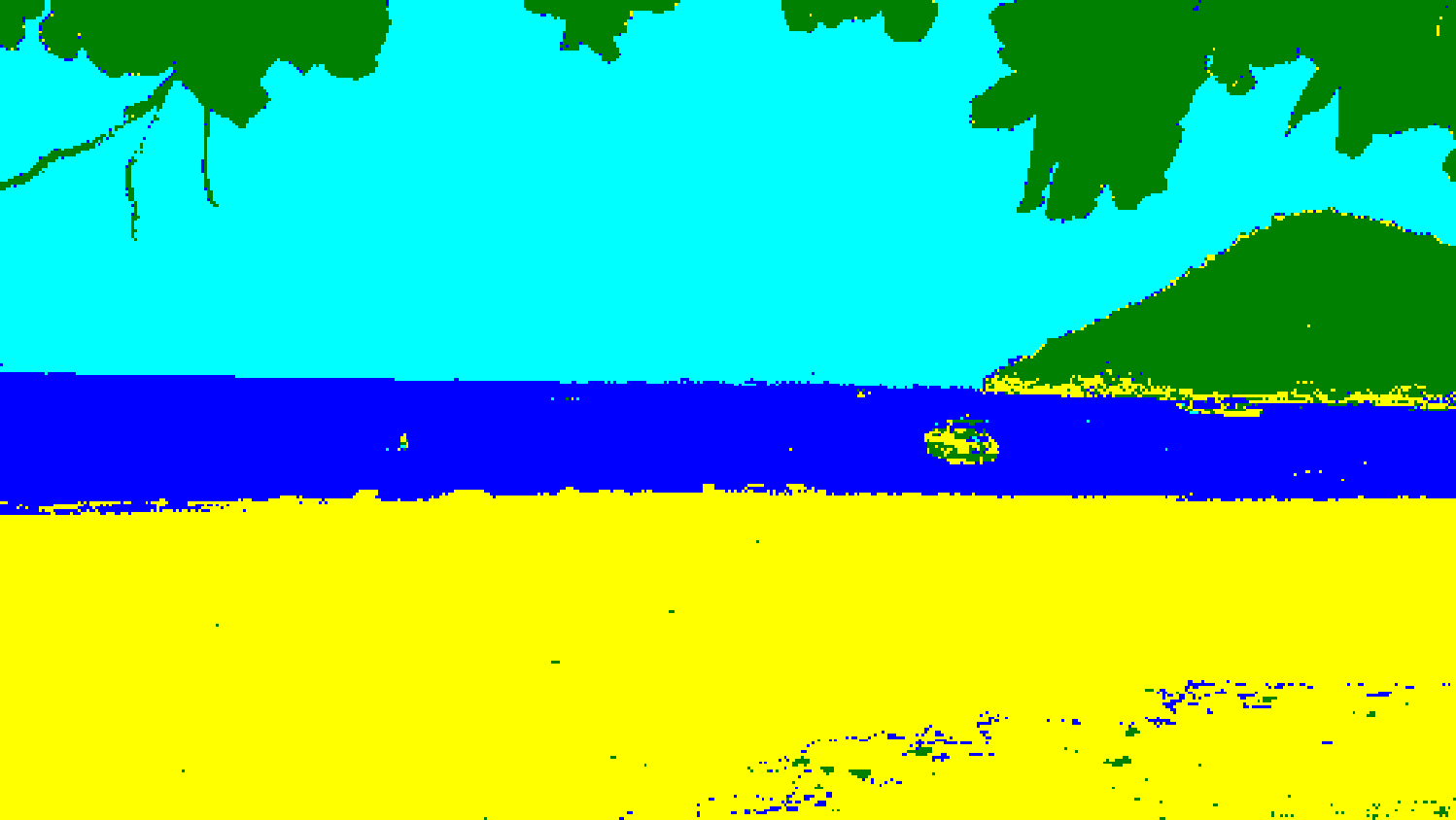}
 	}
 	\subfloat[$xy$ layer]{
 		\includegraphics[width=.49\textwidth]{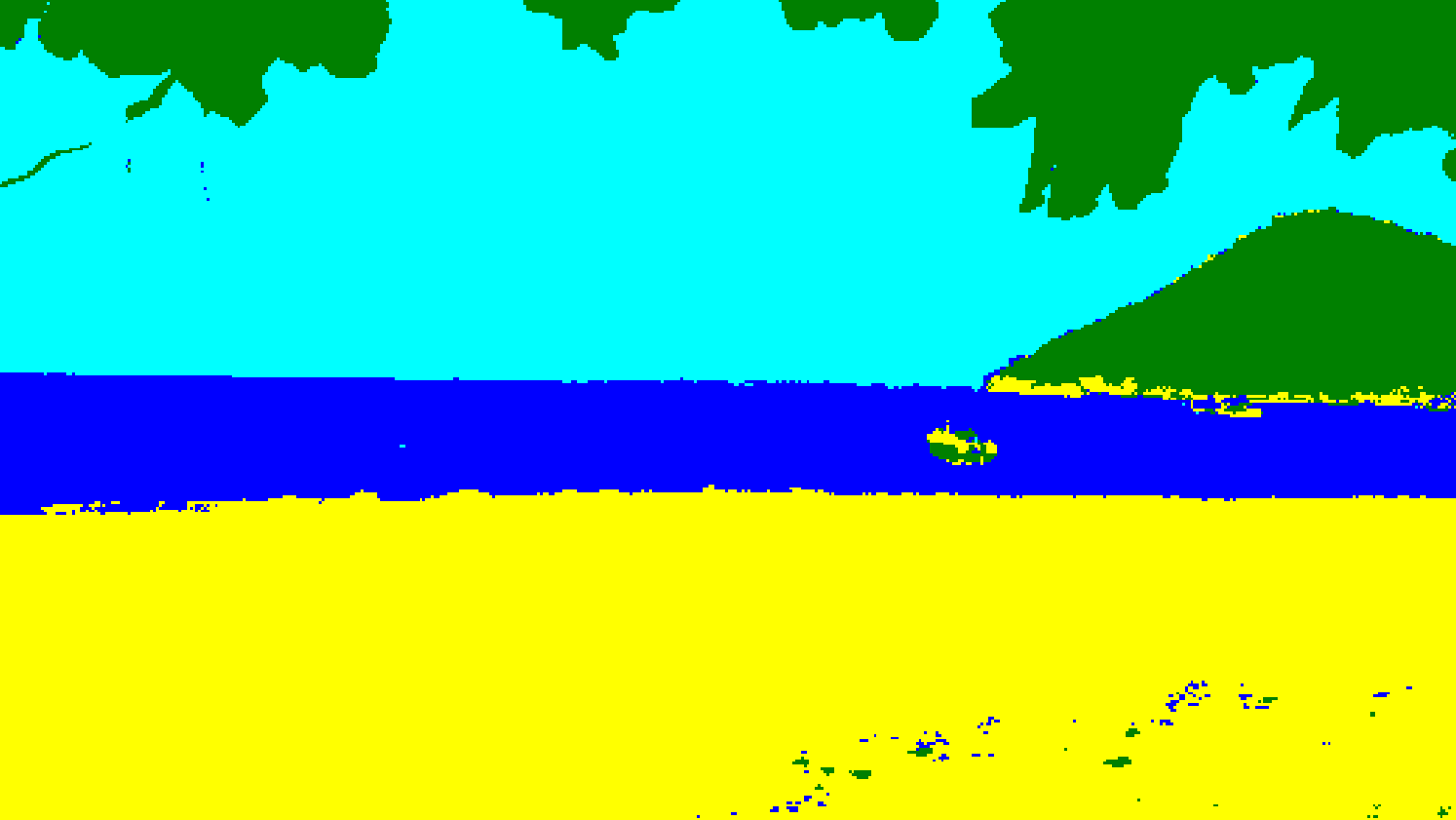}
 	}
 	\caption{\texttt{MPBTV} community detection result for the two layers of the multiplex network corresponding to the image from \Cref{fig:beach_image}.
 		The corresponding multiplex modularity is $0.973$, the classification accuracy $0.976$, and the NMI score $0.909$.}\label{fig:beach_MPBTV_result}
 \end{figure}
 
 \Cref{fig:beach_MPBTV_result} shows the classification result of \texttt{MPBTV} on both layers of the multiplex network corresponding to \Cref{fig:beach_image_im}.
 The $2.4\%$ classification error arises from the systematically misclassified ``swimming object'' pixels as well as some misclassifications on the border between shadow and sand in the ``beach'' class in the bottom right corner of the image.
 The partition shown in \Cref{fig:beach_MPBTV_result} is very similar to results reported in the semi-supervised setting of \cite{bergermann2021semi} albeit having no access to \emph{a priori} class information apart from the choice of a resolution parameter that promotes the detection of a suitable number of communities.
 
 \begin{table}
 	\begin{center}
 		\begin{tabular}{lccc}
 			\hline\hline
 			&Modularity & Accuracy & NMI\\\hline\hline
 			\texttt{MPBTV} & 0.957 & \second{0.969} & \second{0.892}\\
 			\dir & 0.954 & \first{0.973} & \first{0.899}\\\hline
 			GenLouvain \cite{mucha2010community} & \first{0.960} & 0.925 & 0.838\\
 			Leiden MP \cite{traag2019louvain}& \second{0.958} & 0.915 & 0.804\\
 			LART \cite{kuncheva2015community} & 0.929 & 0.380 & 0\\
 			MDLP \cite{boutemine2017mining} & 0.835 & 0.268 & 0.525\\\hline\hline
 		\end{tabular}
 	\end{center}
 	\caption{Comparison of multiplex community detection methods by different metrics for a downscaled version of \Cref{fig:beach_image_im} of size $37\times 65$.}\label{tab:beach_64th_comparison}
 \end{table}
 
 \begin{figure}[t]
 	\centering
 	\subfloat[Offline runtimes in seconds]{
 		\includegraphics[width=0.45\textwidth]{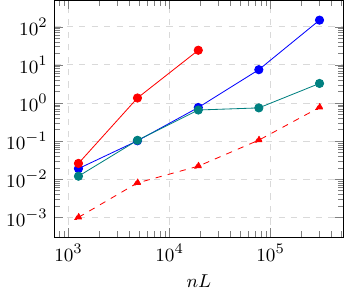}\label{fig:runtimes_beach_image_offline}
%
%
%
%
%
 	}
 	\hfill
 	\subfloat[Runtimes per run in seconds]{
 		\includegraphics[width=0.45\textwidth]{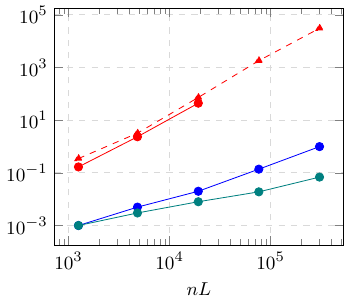}\label{fig:runtimes_beach_image_per_run}
%
%
%
%
%
 	}
 
 	\vspace{5pt}
 	\includegraphics[width=0.8\textwidth]{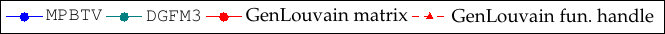}
 	\vspace{-5pt}
 	\caption{Runtimes of the three methods implemented in \texttt{Matlab} in seconds.
 		For GenLouvain, ``matrix'' denotes the version that explicitly assembles the full multiplex modularity matrix $\bm{M}$ while ``fun.\ handle'' denotes the function handle version that provides access to individual columns of $\bm{M}$.
 		Offline runtimes of \texttt{MPBTV} and \dir~correspond to eigenvalue computations.
 		Offline runtimes of GenLouvain matrix correspond to the assembling of the multiplex modularity matrix and that of GenLouvain fun.\ handle to the set-up of the function handle.
 		The assembling of the multiplex modularity matrix in GenLouvain matrix exceeds the available $16$GB of memory for $nL>19\,240$.
 		Runtimes ``per run'' of \texttt{MPBTV} and \dir~correspond to the solution of \eqref{eq:ODE} and \eqref{eq:ODE_dir}, respectively, for one initial condition while for GenLouvain they refer to the execution of the method for one randomized node-layer pair ordering.
 		All runtimes are averaged over $10$ independent code executions (with the exception of the ``per run'' runtime of GenLouvain fun.\ handle for $nL=304\,720$, which is only averaged over $5$ independent code executions for time reasons).}\label{fig:runtimes_beach_image}
 \end{figure}
 
 The size $nL=304\,720$ of the matrices involved in \Cref{fig:beach_MPBTV_result} prohibits the application of the \texttt{python} implementations of Infomap, LART, and MDLP as well as the full matrix version of GenLouvain due to memory constraints and the runtime of the function handle version of GenLouvain is impractically long, cf.~\Cref{fig:runtimes_beach_image}.
 In order to allow a comparison of the methods, we consider a version of \Cref{fig:beach_image} that is downscaled to $37\times 65$ pixels leading to $nL=4\,810$.
 Apart from Infomap, which did not produce valid partitions, we report multiplex modularities, classification accuracies, and NMIs of all methods in \Cref{tab:beach_64th_comparison}.
 Similarly to the results presented in \Cref{sec:numerical_experiments_GT_networks}, \texttt{MPBTV} and \dir~obtain the highest classification accuracies and NMIs albeit ranging slightly below GenLouvain in terms of multiplex modularity.
 GenLouvain classification results across different image resolutions systematically misclassify the shadow on the beach in the bottom-right corner of the image as ``tree''.
 
 Finally, we consider the scaling of the computational complexity of \texttt{MPBTV} and \dir~in comparison with GenLouvain on multiplex networks corresponding to different resolutions of \Cref{fig:beach_image}.
 We include the full matrix version (``GenLouvain matrix''), which explicitly assembles the full multiplex modularity matrix $\bm{M}$ as long as it can be stored in memory as well as the function handle version (``GenLouvain fun.\ handle''), which provides access to individual columns of $\bm{M}$ making it independent of memory restrictions.
 The eigenvalue and -vector computations of \texttt{MPBTV} and \dir~rely on function handles that merely require access to the (sparse) supra-adjacency matrix.
 
 \Cref{fig:runtimes_beach_image_offline} shows offline runtimes that are only required once for any given experiment since they denote the eigenvalue and -vector computations for \texttt{MPBTV} and \dir~as well as the assembling of $\bm{M}$ or the function handle for GenLouvain.
 \Cref{fig:runtimes_beach_image_per_run} shows runtimes per run, i.e., the solution of \eqref{eq:ODE} or \eqref{eq:ODE_dir} for one initial condition $\bm{U}_0$ or the execution of GenLouvain for one randomized node-layer pair ordering, which is required $n_{\mathrm{runs}}$ times.
 \Cref{fig:runtimes_beach_image} shows that the runtime of GenLouvain is dominated by the ``per run'' runtimes while for \texttt{MPBTV} and \dir~the most demanding part are the offline computations.
 Even for the largest considered example with $nL=304\,720$, the ``per run'' runtimes of \texttt{MPBTV} and \dir~only range at or below $1$ second.
 The ``per run'' runtimes of GenLouvain not only range orders of magnitude above those of \texttt{MPBTV} and \dir~for all considered problem sizes, but exhibit an approximately quadratic scaling while \texttt{MPBTV} and \dir~scale almost linearly.
 In the example of \Cref{fig:beach_MPBTV_result}, which corresponds to the right-most dots in \Cref{fig:runtimes_beach_image} with $n_{\mathrm{runs}}=100$, the factor between overall runtimes of \texttt{MPBTV} and GenLouvain is of order $10^5$.

\section{Conclusion and Outlook}
 
 We presented two methods for community detection in multiplex networks based on gradient flows of suitable energy functionals that formally maximize multiplex modularity.
 Partitions obtained by our methods often show higher classification accuracies and NMIs although ranging slightly below multiplex modularities of the best competitor.
 Our efficient numerical treatment allows runtime gains up to several orders of magnitudes.
 
 We observed, however, that our methods struggle to accurately identify relatively small communities.
 Besides addressing this point, future work could investigate, which insights our approach yields into large high-dimensional problems such as hyperspectral imaging data sets.
 Moreover, closely related recent computational approaches \cite{budd2021classification,li2024mbo} for un- and semi-supervised node classification could be generalized to the multiplex case.

\enlargethispage{20pt}

\dataccess{\texttt{Matlab} codes for replicating the numerical experiments are publicly available under \url{https://github.com/KBergermann/GradFlowModMax}.}

\aiuse{We have not used AI-assisted technologies in creating this article.}

\aucontribute{K.B.: Conceptualization; Data curation; Formal analysis; Investigation; Methodology; Software; Visualization; Writing – original draft; M.S.: Conceptualization; Investigation; Methodology; Supervision; Validation; Writing – review \& editing}

\competing{We declare we have no competing interests.}

\funding{Not applicable.}

\ack{We thank Zachary Boyd for sharing the \texttt{Matlab} implementation of the single-layer Balanced TV method \cite{boyd2018simplified}.}


\vskip2pc


%
%
%
%
%
%
%
%
%


\end{document}